\documentclass[11pt]{article}

\usepackage{verbatim,amsfonts,graphicx,amsmath,amssymb,amsthm,geometry,
hyperref,color,nicefrac,enumitem,relsize,titling}
\usepackage[latin1]{inputenc}
\usepackage{bbm} 

\usepackage[nosort,capitalise]{cleveref} 

\crefname{enumi}{item}{items}

\newcommand{\R}{{\mathbb R}}
\newcommand{\N}{{\mathbb N}}

\newcommand{\EE}{\mathbb E}
\newcommand{\bEE}[1]{\mathbb E\bigl[ #1 \bigr]}

\newcommand{\bbbEE}[1]{\mathbb E\biggl[ #1 \biggr]}

\newcommand{\bp}[1]{\bigl(#1\bigr)}
\newcommand{\bbp}[1]{\Bigl(#1\Bigr)}
\newcommand{\bbbp}[1]{\biggl(#1\biggr)}
\newcommand{\bbbbp}[1]{\Biggl(#1\Biggr)}
\newcommand{\br}[1]{[#1]}
\newcommand{\bbr}[1]{\bigl[#1\bigr]}

\newcommand{\bbbbr}[1]{\biggl[#1\biggr]}
\newcommand{\bbbbbr}[1]{\Biggl[#1\Biggr]}
\newcommand{\PP}{{\mathbb P}}

\newcommand{\diff}{\mathrm{d}}
\newcommand{\1}{{\mathbbm 1}}

\newcommand{\eps}{\varepsilon}

\newcommand{\mc}[1]{\mathcal{#1}}
\newcommand{\norm}[1]{\lVert #1\rVert}

\newcommand{\abs}[1]{\lvert #1 \rvert}
\newcommand{\babs}[1]{\bigl\lvert #1 \bigr\rvert}

\def\beq#1\eeq{\begin{equation}#1\end{equation}}
\def\ba#1\ea{\begin{equation}\begin{split}#1\end{split}\end{equation}}
\def\bm#1\em{\begin{multline}#1\end{multline}}

\theoremstyle{plain}
\newtheorem{theorem}{Theorem}[section]
\newtheorem{prop}[theorem]{Proposition}
\newtheorem{lemma}[theorem]{Lemma}

\theoremstyle{definition}

\makeatletter
\newcommand{\nnorm}{\@ifstar\@nnorms\@nnorm}
\newcommand{\@nnorms}[1]{%
  \left|\mkern-1.5mu\left|\mkern-1.5mu\left|
   #1
  \right|\mkern-1.5mu\right|\mkern-1.5mu\right|
}
\newcommand{\@nnorm}[2][]{%
  \mathopen{#1|\mkern-1.5mu#1|\mkern-1.5mu#1|}
  #2
  \mathclose{#1|\mkern-1.5mu#1|\mkern-1.5mu#1|}
}
\makeatother

\newcommand{\qqandqq}{\qquad\text{and}\qquad}

\newcommand{\intrtype}[1]{}
\newcommand{\intrtypen}[1]{#1}
\newcommand{\com}[1]{\ignorespaces}

\newcommand{\is}{\leftarrow}

\title{
  Counterexamples to local Lipschitz and local \\
  H\"older continuity with respect to the initial values \\
  for additive noise driven stochastic differential \\
  equations with smooth drift coefficient functions with \\
  at most polynomially growing derivatives
}

\author{Arnulf Jentzen$^{1,2}$,
  Benno Kuckuck$^{3,4}$,\\
  Thomas M\"uller-Gronbach$^5$, 
  and 
  Larisa Yaroslavtseva$^{6,7}$
  \bigskip
  \\
  \small{$^1$ Department of Mathematics, ETH Zurich, Z\"urich,}\\
  \small{Switzerland, e-mail: arnulf.jentzen@sam.math.ethz.ch}
  \smallskip
  \\
  \small{$^2$ Faculty of Mathematics and Computer Science, University of M\"unster, }\\
  \small{M\"unster, Germany, e-mail: ajentzen@uni-muenster.de}
  \smallskip
  \\
  \small{$^3$ Institute of Mathematics, University of D\"usseldorf, D\"usseldorf,}\\
  \small{Germany, e-mail: kuckuck@math.uni-duesseldorf.de}
  \smallskip\\
  \small{$^4$ Faculty of Mathematics and Computer Science, University of M\"unster, }\\
  \small{M\"unster, Germany, e-mail: bkuckuck@uni-muenster.de}
  \smallskip\\
  \small{$^5$Faculty of Computer Science and Mathematics, 
    University of Passau,}\\
  \small{Passau, Germany, e-mail: thomas.mueller-gronbach@uni-passau.de} 
\smallskip\\
\small{$^6$Faculty of Computer Science and Mathematics, 
University of Passau,}
\\
\small{Passau, Germany, e-mail: larisa.yaroslavtseva@uni-passau.de} 
\smallskip\\
\small{$^7$Faculty of Mathematics and Economics, 
University of Ulm,}
\\
\small{Ulm, Germany, e-mail: larisa.yaroslavtseva@uni-ulm.de} 
}

\begin{document}

\maketitle

\begin{abstract}
  In the recent article [A.~Jentzen, B.~Kuckuck, T.~M\"uller-Gronbach, 
  and L.~Yaroslavtseva, arXiv:1904.05963 (2019)] it has been proved
  that the solutions to every additive noise driven stochastic
  differential equation (SDE) which has a drift coefficient
  function with at most polynomially growing first order partial derivatives
  and which admits a Lyapunov-type condition (ensuring the the existence of
  a unique solution to the SDE) depend in a logarithmically H\"older continuous
  way on their initial values. One might then wonder whether this result
  can be sharpened and whether in fact, SDEs from this
  class necessarily have solutions which depend locally Lipschitz continuously
  on their initial value. The key contribution of this article
  is to establish that this is not the case. More precisely,
  we supply a family of examples of additive noise driven SDEs which
  have smooth drift coefficient functions with at most polynomially
  growing derivatives whose solutions do not depend on their initial
  value in a locally Lipschitz continuous, nor even in a locally H\"older
  continuous way.
\end{abstract}

\tableofcontents

\section{Introduction}


The regularity of stochastic differential equations (SDEs) with respect
to their initial values naturally arises as an important problem in stochastic
analysis (cf., e.g., 
  Chen \& Li \cite{CL14}, 
  Cox et al.\ \cite{CoxHutzenthalerJentzen}, 
  Fang et al.\ \cite{FIZ07},
  Hairer et al.\ \cite{HHJ15},
  Hairer \& Mattingly \cite{HairerMattingly},
  Hudde et al.\ \cite{HHM19},
  Krylov \cite{Kry99},
  Li \cite{Li94},
  Li \& Scheutzow \cite{LS11}, 
  Liu \& R\"ockner \cite{LiuRoeckner},
  and Scheutzow \& Schulze \cite{SS17}). 
At the same time this problem has strong links to the analysis
of numerical approximations for SDEs (cf., e.g., 
Hudde et al.\ \cite{HuddeHutzenthalerJentzenMazzonetto},
Hutzenthaler \& Jentzen \cite{HutzenthalerJentzen},
Hutzenthaler et al.\ \cite{hutzenthaler2019strong},
and Zhang \cite{Zhang10}%
).
There are several results in the scientific literature that provide sufficient
or necessary conditions which ensure that SDEs have suitable regularity 
properties in the initial value (cf., e.g., 
\cite{CL14,CoxHutzenthalerJentzen,FIZ07,HHJ15,HairerMattingly,HHM19,Kry99,Li94,LS11,LiuRoeckner,SS17}%
).
%
%
%
In particular, in the recent article 
\cite{JentzenKuckuckMuellerGronbachYaroslavtseva2019} 
it has been proved that 
every additive noise driven SDE which admits a Lyapunov-type condition 
(that ensures the existence of a unique solution of the SDE) 
and which has a drift coefficient function whose first order 
partial derivatives grow at most polynomially is at least 
logarithmically H\"older continuous in the initial value.
This result shows that the solutions of additive noise driven SDEs which have a smooth
drift coefficient function with at most polynomially growing derivative
cannot have arbitrarily bad regularity properties with respect to the
initial value (cf., e.g., the negative results in Hairer et al.\ \cite{HHJ15}
and \cite{JMGY16}
for SDEs without the restriction on the drift
coefficient function).
However, this result does not imply local Lipschitz continuity
with respect to the initial value.
Having this in mind, one may wonder whether the main result in
\cite{JentzenKuckuckMuellerGronbachYaroslavtseva2019}
is actually sharp or whether, in fact, SDEs from this
class necessarily have solutions which depend locally 
Lipschitz continuously on their initial value.
The key contribution of this article is to establish that 
this is not the case.
More precisely, the main result of this article,
\cref{thm:final} in Subsection~\ref{subsec:final} below,
shows that
there are additive noise driven SDEs which have smooth drift coefficient functions
with at most polynomially growing derivatives whose solutions
do not depend on their initial value in a locally Lipschitz continuous, 
nor even in a locally H\"older continuous way.
In order to illustrate the findings of this article
in more detail, we now present in the following
theorem a simplified version of \cref{thm:final} below.
\begin{theorem}
  \label{thm:intro}
  Let $m\in\N$,
  $d\in\{5,6,\dots\}$,
  $T\in(0,\infty)$,
  $\tau\in(0,T)$,
  $v\in\R^d$,
  $\delta\in\R^d\setminus\{0\}$
  let $\norm{\cdot}\colon\R^d\to[0,\infty)$ be the standard norm on $\R^d$,
  let $\nnorm{\cdot}\colon\R^m\to[0,\infty)$ be a norm,
  let $(\Omega,\mc F,\PP)$ be a probability space,
  and let $W\colon[0,T]\times\Omega\to\R^m$ be a standard Brownian motion
    with continuous sample paths.
  Then there exist
    $\mu\in C^\infty(\R^d,\R^d)$,
    $\sigma\in\R^{d\times m}$,
    $V\in C^\infty(\R^d,[0,\infty))$,
    $\kappa\in(0,\infty)$
  such that
  \begin{enumerate}[label=(\roman{enumi}),ref=(\roman{enumi})]
    \item \label{it:intro:1}
    it holds for all
      $x,h\in\R^d$, 
      $z\in\R^m$ 
    that
      $\norm{\mu'(x)h}\leq \kappa \bp{1+\norm x^\kappa} \norm h$,
      $V'(x)\mu(x+\sigma z)\leq \kappa(1+\nnorm{z})V(x)$,
      and $\norm x\leq V(x)$,
    \item \label{it:intro:2}
      there exist unique stochastic processes
      $X^x\colon [0,T]\times\Omega\to\R^d$, $x\in\R^d$, with continuous sample paths
      such that for all 
        $x\in\R^d$,
        $t\in[0,T]$,
        $\omega\in\Omega$
      it holds that
      \beq
        X^x(t,\omega)=x+\int_0^t\mu(X^x(s,\omega))\,\diff s+\sigma W(t,\omega)
        ,
      \eeq
    \item \label{it:intro:3}
      it holds for all 
        $R,r\in(0,\infty)$
      that
      \beq
        \sup_{x\in[-R,R]^d}\,\sup_{t\in[0,T]} \bEE{ \norm{X^x(t)}^{r} }<\infty
        ,
      \eeq
    \item \label{it:intro:4}
      it holds for all 
        $R,q\in(0,\infty)$
      that there exists 
        $c\in(0,\infty)$ 
      such that for all 
        $x,y\in[-R,R]^d$ with $0<\norm{x-y}\neq 1$ 
      it holds that
      \beq
        \sup_{t\in[0,T]}\bEE{\norm{X^x(t)-X^y(t)}}\leq c\,\abs{\ln(\norm{x-y})}^{-q}
        ,
      \eeq
      and
    \item \label{it:intro:5}
      it holds for all 
        $t\in(\tau,T)$,
        $\alpha\in(0,\infty)$
      that there exists 
        $c\in(0,\infty)$ 
      such that for all
        $w\in\{v+r\delta\colon r\in[0,1]\}$
      it holds that
      \beq
        c\,\norm{v-w}^\alpha
        \leq
        \bEE{ \norm{X^v(t)-X^w(t)} } 
        .
      \eeq
  \end{enumerate}
\end{theorem}

\Cref{thm:intro} above is an immediate consequence of \cref{thm:final} below,
the main result of this article. \Cref{thm:final} in turn is proved by explicitly
constructing a specific example of a family of SDEs with the desired properties 
(cf., e.g., \eqref{eq:failhoelderconc}, 
\eqref{eq:exist0.defmu}, 
\eqref{eq:lemexistconc},
\eqref{eq:lem1nonhoelder},
and \eqref{eq:finalnonhoelder})
Observe that \cref{thm:intro} establishes the existence of an additive noise
driven SDE with a smooth drift coefficient function 
$\mu\colon \R^d\to\R^d$ and a diffusion coefficient $\sigma\in\R^{d\times m}$
such that the drift coefficient function has at most polynomially growing derivatives
and admits a suitable Lyapunov-type condition (see \cref{it:intro:1} above), 
such that there exist unique solution processes $X^x\colon [0,T]\times \Omega\to\R^d$,
$x\in\R^d$,
to this SDE (see \cref{it:intro:2} above),
such that for every $x\in\R^d$, $t\in[0,T]$ the absolute moments
$\bEE{\norm{X^x(t)}^r}$, $r\in(0,\infty)$, of the solution processes are 
finite (see \cref{it:intro:3} above),
and such that the solution is regular with respect to the initial value in 
the sense of \cref{it:intro:4} above (cf.\ \cite[Theorem~8.4]{JentzenKuckuckMuellerGronbachYaroslavtseva2019}), yet fails
to be locally Lipschitz or locally H\"older continuous
in the initial values (see \cref{it:intro:5} above).

The remainder of this article is organized as follows:
In \cref{sec:axis-aligned} we establish, roughly speaking, 
the existence of additive noise driven SDEs whose solutions depend non-locally H\"older 
continuously on their initial values.
In \cref{sec:axis-aligned-existence} we establish, roughly speaking, 
the existence of solutions to certain additive noise driven SDEs whose 
solutions depend 
non-locally H\"older continuously on their initial values and whose
drift coefficient functions are smooth with at most polynomially 
growing derivatives.
In \cref{sec:general} we combine the results of
\cref{sec:axis-aligned-existence} with an
essentially well-known fact
on affine linear transformations of solutions to SDEs
in order to prove \cref{thm:final}, the main result of this article.


  
  
  

\section{On the existence of axis-aligned stochastic differential equations (SDEs) with non-locally H\"older continuous
dependence on the initial values}
\label{sec:axis-aligned}

In this section we establish in \cref{prop:failhoelder} below, roughly speaking, 
the existence of additive noise driven  SDEs whose solutions depend non-locally H\"older continuously on their initial values.
In our proof of \cref{prop:failhoelder} we employ the
elementary lower bound for certain functionals of standard normal random
variables in \cref{lem:stricthelp} below as well as the essentially well-known result on
certain Lebesgue integrals involving standard Brownian motions in \cref{lem:stdnorm} below.
For completeness we also provide in this section a detailed proof for
\cref{lem:stdnorm}. 
Our proof is elementary and avoids 
stochastic integration theory and the use of It\^o's formula.
Alternatively, \cref{lem:stdnorm} can be proved through a
straightforward application of It\^o's lemma.

\subsection{Lower bounds for certain functionals of standard normal random variables}

\begingroup
\newcommand{\vp}{p}
\begin{lemma}
  \label{lem:stricthelp}
  Let $p\in[1,\infty)$,
  $\kappa,c\in(0,\infty)$,
  $\eps\in(0,\nicefrac 1e]$
  satisfy
  $
    c=p\kappa^{-2/p}+([2\pi]^{1/2}p+1)\kappa+1
  $,
  let $(\Omega,\mc F,\PP)$ be a probability space,
  and let $Z\colon \Omega\to\R$ be a standard normal random variable.
  Then
  \beq
  \label{eq:stricthelpconc}
    \bEE{\eps\exp\bp{\kappa\,\abs{Z}^p-\eps^2\kappa\exp(2\kappa\,\abs{Z}^p)}}
    \geq 
    \exp\bp{ -c\,\abs{\ln(\eps)}^{2/p} }.
  \eeq
\end{lemma}
\begin{proof}[Proof of \cref{lem:stricthelp}]
  Throughout this proof 
  let $\psi\colon\R\to(0,\infty)$ 
    \intrtype{be the function which satisfies }%
    \intrtypen{satisfy }%
    for all 
      $z\in\R$
    that
    \beq
    \label{eq:stricthelp.defpsi}
      \psi(z)=\exp\bp{ z-\kappa\exp(2z) }
    \eeq
  and let $a,b\in[0,\infty)$
    \intrtype{be the real numbers which satisfy }%
    \intrtypen{satisfy }%
  \beq
  \label{eq:abdef}
    \textstyle
    a=\bigl[\kappa^{-1}(\ln(\nicefrac1\eps)-1)\bigr]^{1/p}
    \qqandqq
    b=\bigl[\kappa^{-1}\ln(\nicefrac1\eps)\bigr]^{1/p}
    .
  \eeq
  Note that 
    \eqref{eq:stricthelp.defpsi}
  ensures that for all 
    $z\in\R$ 
  it holds that
  \ba
    \psi\bp{\kappa\, \abs{z}^p+\ln(\eps)}
    &=\exp\bp{ \kappa\, \abs{z}^p+\ln(\eps)-\kappa\exp\bp{2\kappa\,\abs{z}^p+2\ln(\eps)} }\\
    &=\eps\exp\bp{ \kappa\, \abs{z}^p-\eps^2\kappa\exp\bp{2\kappa\,\abs{z}^p} }
    .
  \ea
  Combining
    this
  with
    the hypothesis that 
      $Z$ is a standard normal random variable,
    the fact that 
      $0\leq a<b$,
    and the fact that 
      $[0,\infty)\ni z\mapsto \exp(-z^2/2)\in(0,\infty)$ is a decreasing function
  shows that
  \ba
  \label{eq:bnddexp}
    \bEE{\eps\exp\bp{\kappa\,\abs{Z}^p-\eps^2\kappa\exp(2\kappa\, \abs{Z}^p)}}
    &=
    \bEE{\psi(\kappa\,\abs{Z}^p+\ln(\eps))}
    \\&\geq 
    \frac1{\sqrt{2\pi}}\int_{a}^{b}\psi(\kappa\, \abs{z}^p+\ln(\eps))\exp\bbbp{\frac{-z^2}2}\,\diff z
    \\&\geq 
    \frac{\exp\bp{\frac{-b^2}2}}{\sqrt{2\pi}}\int_{a}^{b}\psi(\kappa\,\abs{z}^p+\ln(\eps))\,\diff z
    \\&\geq 
    \frac{\exp\bp{\frac{-b^2}2}(b-a)}{\sqrt{2\pi}}\bbbbr{ \inf_{z\in[a,b]}\psi(\kappa\,\abs{z}^p+\ln(\eps)) }
    .
  \ea
  Next observe that 
    the fact that
      $\{\kappa\,\abs{z}^p+\ln(\eps)\colon z\in [a,b]\}=[-1,0]$
  implies that
  \beq
  \label{eq:infpsi}
    \inf_{z\in[a,b]}\psi(\kappa\,\abs{z}^p+\ln(\eps))
    =
    \inf_{y\in[-1,0]}\psi(y)
    =
    \exp\bbbp{\inf_{y\in[-1,0]} (y-\kappa\exp(2y))}
    \geq
    \exp(-1-\kappa)
    .
  \eeq
  In the next step we note that
    the fact that
      for all
        $z\in(0,\infty)$
      it holds that
        $\ln(z)\leq z$
  proves that for all
    $z\in(0,\infty)$
  it holds that
  \beq
    \exp\bbbp{-\frac{(p-1)z^2}2}
    \leq 
    \exp\bbbp{-\frac{(p-1)\ln(z^2)}2}
    =
    (z^2)^{-\frac{p-1}2}
    =
    z^{-(p-1)}
    .
  \eeq
    The fundamental theorem of calculus,
    the fact that $0\leq a^p<b^p$,
    and the fact that
      $(0,\infty)\ni z\mapsto z^{-(p-1)/p}\in\R$ is a decreasing function
    therefore
  ensure that
  \ba
    b-a
    &=
    \bbr{z^{\frac1p}}_{z=a^p}^{z=b^p}
    =
    p^{-1}\bbbbr{\int_{a^p}^{b^p}z^{-\frac{p-1}p}\,\diff z}
    \geq 
    p^{-1}(b^p-a^p)\bbbbr{\inf_{z\in[a^p,b^p]}z^{-\frac{p-1}p}}\\
    &= 
    p^{-1}\kappa^{-1}b^{-(p-1)}
    \geq 
    (p\kappa)^{-1}\exp\bbbp{-\frac{(p-1)b^2}{2}}
    .
  \ea
    This
    and the fact that for all 
      $x\in(0,\infty)$ 
    it holds that 
      $\exp (-x)\leq \frac1x$ 
  show that
  \ba
    \frac{\exp\bp{\frac{-b^2}2}(b-a)}{\sqrt{2\pi}}
    &\geq
    \bbr{(2\pi)^{1/2}p\kappa}^{-1}\exp\bbbp{-\frac{pb^2}2}
    \\&=
    \bbr{(2\pi)^{1/2}p\kappa}^{-1}\exp\bbbp{-\frac{p\bigl[\kappa^{-1}\ln(\nicefrac1\eps)\bigr]^{2/p}}{2}}
    \\&\geq
    \bbr{(2\pi)^{1/2}p\kappa}^{-1}\exp\bp{-p\kappa^{-2/p}[\ln(\nicefrac1\eps)]^{2/p}}
    \\&\geq
    \exp\bp{-(2\pi)^{1/2}p\kappa}\exp\bp{-p\kappa^{-2/p}[\ln(\nicefrac1\eps)]^{2/p}}
    \\&=
    \exp\bp{-(2\pi)^{1/2}p\kappa-p\kappa^{-2/p}[\ln(\nicefrac1\eps)]^{2/p}}
    .
  \ea
  Combining 
    this 
  with 
    \eqref{eq:bnddexp} 
    and \eqref{eq:infpsi} 
  demonstrates that
  \ba
  \label{eq:finbnd1}
    &\bEE{\eps\exp\bp{\kappa\,\abs{Z}^p-\eps^2\kappa\exp(2\kappa\,\abs{Z}^p)}}
    \\&\geq 
    \exp\bp{-[2\pi]^{1/2}p\kappa-p\kappa^{-2/p}[\ln(\nicefrac1\eps)]^{2/p}}
    \exp(-1-\kappa)
    \\&=
    \exp\bp{-\bp{[2\pi]^{1/2}p\kappa+1+\kappa+p\kappa^{-2/p}[\ln(\nicefrac1\eps)]^{2/p}}}
    .
  \ea
  The fact that
    $\ln(\nicefrac1\eps)^{2/p}>1$
    hence
  implies that
  \beq
    \bEE{\eps\exp\bp{\kappa\,\abs{Z}^p-\eps^2\kappa\exp(2\kappa\,\abs{Z}^p)}}
    \geq
    \exp\bp{-\bp{[2\pi]^{1/2}p\kappa+1+\kappa+p\kappa^{-2/p}}[\ln(\nicefrac1\eps)]^{2/p}}
    .
  \eeq
  The proof of \cref{lem:stricthelp} is thus completed.
\end{proof}
\endgroup

\subsection{On the distribution of certain integrals involving Brownian motions}

\begin{lemma}
\label{lem:stdnorm}
  Let $\tau\in(0,\infty)$,
  let $(\Omega,\mc F,\PP)$ be a probability space,
  let $g\in C^1(\R,[0,\infty))$ satisfy 
    $\{t\in\R\colon g(t)>0\}\subseteq[0,\tau]$
    and $\int_0^\tau \abs{g(t)}^2\,\diff t=1$,
  let $W\colon[0,\tau]\times \Omega\to\R$ be a standard Brownian motion with continuous sample paths,
  and let $X\colon\Omega\to\R$ be a random variable
    which satisfies for all
      $\omega\in\Omega$
    that
    \beq
    \label{eq:XintW}
      X(\omega)
      =
      \int_0^\tau g'(s)W(s,\omega)\,\diff s
      .
    \eeq
  Then $X$ is a standard normal random variable.
\end{lemma}
\begin{proof}[Proof of \cref{lem:stdnorm}]
  Throughout this proof 
  let $A\subseteq\R^2$
    \intrtype{be the set which satisfies }%
    \intrtypen{satisfy }%
    \beq
      A=\{(x,y)\in\R^2\colon x<y\}
      .
    \eeq
  Observe that 
    \eqref{eq:XintW}
  implies that for all
    $\omega\in\Omega$
  it holds that
  \ba
    \abs{X(\omega)}^2
    &=
    \bbbp{\int_0^\tau g'(s)W(s,\omega)\,\diff s}\bbbp{\int_0^\tau g'(u)W(u,\omega)\,\diff u}
    \\&=
    \int_0^\tau \int_0^\tau g'(s)g'(u)W(s,\omega)W(u,\omega)\,\diff u\,\diff s.
  \ea
    Fubini's theorem
    hence
  shows that
  \ba
    \bEE{\abs{X}^2}
    &=
    \int_0^\tau \int_0^\tau g'(s)g'(u)\EE[W(s)W(u)]\,\diff u\,\diff s\\
    &=
    \int_0^\tau \int_0^\tau g'(s)g'(u)\min(s,u)\,\diff u\,\diff s\\
    &=
    \int_0^\tau \int_0^\tau g'(s)g'(u)\min(s,u)\bp{\1_A(s,u)+\1_{\R^2\setminus A}(s,u)}\,\diff u\,\diff s\\
    &=
    \int_0^\tau\int_0^\tau g'(s)g'(u)\min(s,u)\1_A(s,u)\,\diff u\,\diff s
      \\&\qquad+\int_0^\tau\int_0^\tau g'(s)g'(u)\min(s,u)\1_{\R^2\setminus A}(s,u)\,\diff u\,\diff s
    .
  \ea
    This
    and Fubini's theorem
  demonstrate that
  \ba
  \label{eq:EXsq}
    \bEE{\abs{X}^2}
    &=
    \int_0^\tau\int_0^\tau g'(u)g'(s)\min(s,u)\1_A(s,u)\,\diff s\,\diff u\\
      &\qquad+\int_0^\tau\int_0^\tau g'(s)g'(u)\min(s,u)\1_{\R^2\setminus A}(s,u)\,\diff u\,\diff s\\
    &=
    \int_0^\tau \int_0^u g'(u)g'(s)s\,\diff s\,\diff u+\int_0^\tau\int_0^s g'(s)g'(u)u\,\diff u\,\diff s\\
    &=
    2\bbbbr{\int_0^\tau \int_0^s g'(s)g'(u)u\,\diff u\,\diff s}
    =
    2\bbbp{\int_0^\tau g'(s)\bbbbr{\int_0^s g'(u)u\,\diff u}\,\diff s}
    .
  \ea
  Furthermore, note that
    the hypothesis that 
      $\{t\in\R\colon g(t)>0\}\subseteq[0,\tau]$
    and the hypothesis that $g$ is a continuous function
  ensure that 
  \beq
    g(0)=0=g(\tau)
    .
  \eeq
  Combining
    this,
    integration by parts,
    and the hypothesis that $\int_0^\tau \abs{g(s)}^2\,\diff s=1$
  with
    \eqref{eq:EXsq}
  proves that
  \ba
    \bEE{\abs{X}^2}
    &=
    2\bbbp{\bbbbr{g(s)\bbbp{\int_0^s g'(u)u\,\diff u}}_{s=0}^{s=\tau}-\int_0^\tau g(s)g'(s)s\,\diff s}
    \\&=
    -2\bbbbr{\int_0^\tau g(s)g'(s)s\,\diff s}
    =
    -\int_0^\tau \bp{\tfrac\partial{\partial s}\bbr{(g(s))^2}}s\,\diff s
    \\&=
    -\bbbp{\bbr{(g(s))^2s}_{s=0}^{s=\tau}-\int_0^\tau (g(s))^2\,\diff s}
    =
    \int_0^\tau (g(s))^2\,\diff s
    =
    1
    .
  \ea
  Combining
    this
    with \eqref{eq:XintW} 
  establishes that
    $X$ is a standard normal random variable.
  The proof of \cref{lem:stdnorm} is thus completed.
\end{proof}

\subsection{On SDEs with irregularities in the initial value}

\begin{prop}
  \label{prop:failhoelder}
  Let $T\in(0,\infty)$,
  $\tau\in(0,T)$,
  $n\in\N$,
  $f\in C^1(\R,[0,\infty))$,
  $g\in C^2(\R,[0,\infty))$
  satisfy
    $\{t\in\R\colon g(t)>0\}\subseteq [0,\tau]$,
    $\{t\in\R\colon f(t)>0\}=(\tau,T)$,
    and $\int_0^\tau \abs{g(s)}^2\,\diff s=1$,
  let $\sigma=(0,1,0,0,0)\in\R^5$,
  let $\norm{\cdot}\colon\R^5\to[0,\infty)$ be the standard norm on $\R^5$,
  let $\mu\colon\R^5\to\R^5$ satisfy for all
  $x=(x_1,x_2,x_3,x_4,x_5)\in\R^5$ that
  \beq
  \label{eq:mudef}
    \mu(x)=\bp{ 1,0,g'(x_1)x_2,f(x_1)x_4x_5,f(x_1)((x_3)^n-(x_4)^2) },
  \eeq
  let $(\Omega,\mathcal F,\PP)$ be a probability space,
  let $W\colon[0,T]\times\Omega\to\R$ be a
  standard Brownian motion 
    with continuous sample paths,
  and let
    $X^x=(X^x_1,X^x_2,X^x_3,X^x_4,X^x_5)\colon[0,T]\times\Omega\to\R^5$, $x\in\R^5$, 
    be stochastic processes
    with continuous sample paths 
  which satisfy for all
    $x\in\R^5$,
    $t\in[0,T]$,
    $\omega\in\Omega$
  that
  \beq
  \label{eq:Xeq}
    X^x(t,\omega)
    =
    x+\int_0^t \mu(X^x(s,\omega))\,\diff s + \sigma W(t,\omega)
    .
  \eeq
  Then it holds for all 
    $t\in(\tau,T)$ 
  that 
    there exists
      $c\in(0,\infty)$ 
    such that 
      for all
        $\eps\in(0,\nicefrac1e]$,
        $h=(0,0,0,\eps,0)\in\R^5$
      it holds that
      \beq
      \label{eq:failhoelderconc}
        \bEE{ \abs{X_4^h(t)-X_4^0(t)} } \geq \babs{ \EE[X_4^h(t)]-\EE[X_4^0(t)] }
        \geq \exp\bp{ -c\,\abs{\ln(\eps)}^{2/n} }
        .
      \eeq
\end{prop}
\begin{proof}[Proof of \cref{prop:failhoelder}]
  Throughout this proof
    let $v=(0,0,0,1,0)\in\R^5$
    and let $\kappa_t\in[0,\infty)$, $t\in[\tau,T]$, satisfy
      for all $t\in[\tau,T]$
    that
    \beq
      \kappa_t=\int_\tau^t\int_\tau^s f(u)f(s)\,\diff u\,\diff s.
    \eeq
  Observe that 
    \eqref{eq:mudef} 
    and \eqref{eq:Xeq}
  imply that for all 
    $\eps\in[0,\infty)$,
    $t\in[0,T]$,
    $\omega\in\Omega$
  it holds that
  \ba
  \label{eq:X12}
    X^{\eps v}_1(t,\omega)&=\int_0^t 1\,\diff s=t\\
    \text{and}\qquad X^{\eps v}_2(t,\omega)&=\int_0^t 0\,\diff s + W(t,\omega)=W(t,\omega)
    .
  \ea
    This,
    \eqref{eq:mudef},
    and \eqref{eq:Xeq}
  show that for all 
    $\eps\in[0,\infty)$,
    $t\in[0,T]$,
    $\omega\in\Omega$
  it holds that 
  \beq
  \label{eq:X32}
    X_3^{\eps v}(t,\omega)
    =
    \int_0^t g'(X^{\eps v}_1(s,\omega))X^{\eps v}_2(s,\omega)\,\diff s
    =
    \int_0^t g'(s)W(s,\omega)\,\diff s
    .
  \eeq
    The hypothesis that 
      $\{t\in\R\colon g(t)>0\}\subseteq[0,\tau]$
    hence
  ensures that for all 
    $\eps\in[0,\infty)$,
    $t\in[\tau,T]$,
    $\omega\in\Omega$
  it holds that
  \beq
  \label{eq:X3}
    X^{\eps v}_3(t,\omega)
    =
    \int_0^\tau g'(s)W(s,\omega)\,\diff s
    =
    X^{\eps v}_3(\tau,\omega)
    .
  \eeq
  Next note that 
    \eqref{eq:mudef},
    \eqref{eq:Xeq},
    and \eqref{eq:X12} 
  prove that for all 
    $\eps\in[0,\infty)$,
    $t\in[0,T]$,
    $\omega\in\Omega$
  it holds that
  \ba
    X^{\eps v}_4(t,\omega)
    &=
    \eps+\int_0^t f(X^{\eps v}_1(s,\omega))X^{\eps v}_4(s,\omega)X^{\eps v}_5(s,\omega)\,\diff s\\
    &=
    \eps+\int_0^t f(s)X^{\eps v}_4(s,\omega)X^{\eps v}_5(s,\omega)\,\diff s
    .
  \ea
    Hence,
  we obtain that for all 
    $\eps\in[0,\infty)$,
    $t\in[0,T]$,
    $\omega\in\Omega$
  it holds that
  \beq
    X^{\eps v}_4(t,\omega)=\eps\exp\bbbp{\int_0^t f(s) X^{\eps v}_5(s,\omega)\,\diff s }
    .
  \eeq
    The hypothesis that
      $\{t\in\R\colon f(t)>0\}=(\tau,T)$
    hence
  implies that for all 
    $\eps\in[0,\infty)$,
    $t\in[\tau,T]$,
    $\omega\in\Omega$
  it holds that
  \beq
  \label{eq:X4}
    X^{\eps v}_4(t,\omega)=\eps\exp\bbbp{\int_\tau^t f(s) X^{\eps v}_5(s,\omega)\,\diff s }
    .
  \eeq
  Moreover, observe that
    \eqref{eq:mudef},
    \eqref{eq:Xeq},
    and \eqref{eq:X12}
  show that for all 
    $\eps\in[0,\infty)$,
    $s\in[0,T]$,
    $\omega\in\Omega$
  it holds that
  \ba
    X^{\eps v}_5(s,\omega)
    &=
    \int_0^s f(X^{\eps v}_1(u,\omega))\bp{[X^{\eps v}_3(u,\omega)]^n-[X^{\eps v}_4(u,\omega)]^2}\,\diff u
    \\&=
    \int_0^s f(u)\bp{ [X^{\eps v}_3(u,\omega)]^n-[X^{\eps v}_4(u,\omega)]^2 }\,\diff u
    .
  \ea
  Combining
    this
  with
    the hypothesis that
      $\{t\in\R\colon f(t)>0\}=(\tau,T)$
    and \eqref{eq:X3}
  demonstrates that for all 
    $\eps\in[0,\infty)$,
    $s\in[\tau,T]$,
    $\omega\in\Omega$
  it holds that
  \ba
  \label{eq:X5}
    X^{\eps v}_5(s,\omega)
    &=
    \int_\tau^s f(u)\bp{ [X^{\eps v}_3(u,\omega)]^n-[X^{\eps v}_4(u,\omega)]^2 }\,\diff u
    \\&=
    \int_\tau^s f(u)\bp{ [X^{\eps v}_3(\tau,\omega)]^n-[X^{\eps v}_4(u,\omega)]^2 }\,\diff u
    \\&=
    [X^{\eps v}_3(\tau,\omega)]^n\bbbbr{\int_\tau^s f(u)\,\diff u}-\bbbbr{\int_\tau^s f(u)[X^{\eps v}_4(u,\omega)]^2\,\diff u}
    .
  \ea
    This
    and \eqref{eq:X4} 
  prove that for all
    $\eps\in[0,\infty)$,
    $t\in[\tau,T]$,
    $\omega\in\Omega$
  it holds that
  \begin{align}
  \label{eq:X4.2}
    \nonumber &X^{\eps v}_4(t,\omega)\\
    &=
    \nonumber \eps\exp\bbbp{\int_\tau^t  f(s) \bbbbr{[X^{\eps v}_3(\tau,\omega)]^n\bbbp{\int_\tau^s f(u)\,\diff u}-\bbbp{\int_\tau^s f(u)[X^{\eps v}_4(u,\omega)]^2\,\diff u}}\,\diff s }\\
    &=
    \nonumber \eps\exp\bbbp{\int_\tau^t f(s) [X^{\eps v}_3(\tau,\omega)]^n\bbbp{\int_\tau^s f(u)\,\diff u}\diff s-\int_\tau^tf(s)\bbbp{\int_\tau^s f(u)[X^{\eps v}_4(u,\omega)]^2\,\diff u}\diff s }\\
    &=
    \eps\exp\bbbp{[X^{\eps v}_3(\tau,\omega)]^n\int_\tau^t  \int_\tau^s f(u)f(s)\,\diff u\,\diff s-\int_\tau^t\int_\tau^s f(s)f(u)[X^{\eps v}_4(u,\omega)]^2\,\diff u\,\diff s }\\
    &=
    \nonumber \eps\exp\bbbp{\kappa_t[X^{\eps v}_3(\tau,\omega)]^n-\int_\tau^t\int_\tau^s f(s)f(u)[X^{\eps v}_4(u,\omega)]^2\,\diff u\,\diff s }
    .
  \end{align}
    Therefore,
  we obtain that for all 
    $\eps\in[0,\infty)$,
    $u\in[\tau,T]$,
    $\omega\in\Omega$
  it holds that
  \beq
    \label{eq:bndX10}
    X^{\eps v}_4(u,\omega)
    \leq
    \eps\exp\bp{\kappa_u[X^{\eps v}_3(\tau,\omega)]^n}
    .
  \eeq
  Combining 
    this 
  with
    the fact that for all
      $s\in[\tau,T]$,
      $t\in[s,T]$
    it holds that
    $
      \kappa_t\geq \kappa_s
    $
  ensures that for all 
    $\eps\in[0,\infty)$,
    $t\in[\tau,T]$,
    $\omega\in\Omega$
  it holds that
  \ba
    &\int_\tau^t\int_\tau^s f(s)f(u)[X^{\eps v}_4(u,\omega)]^2\,\diff u\,\diff s
    \leq
    \int_\tau^t\int_\tau^s f(s)f(u)\eps^2\exp\bp{2\kappa_u[X^{\eps v}_3(\tau,\omega)]^n}\,\diff u\,\diff s
    \\&\leq
    \eps^2\exp\bp{2\kappa_t[X^{\eps v}_3(\tau,\omega)]^n}\int_\tau^t\int_\tau^s f(s)f(u)\,\diff u\,\diff s
    =
    \eps^2\kappa_t\exp\bp{2\kappa_t[X^{\eps v}_3(\tau,\omega)]^n}
    .
  \ea
    This 
    and \eqref{eq:X4.2}
  establish that for all 
    $\eps\in[0,\infty)$,
    $t\in[\tau,T]$,
    $\omega\in\Omega$
  it holds that
  \beq
  \label{eq:bndX4}
    X^{\eps v}_4(t,\omega)
    \geq
    \eps\exp\bbp{\kappa_t[X^{\eps v}_3(\tau,\omega)]^n-\eps^2\kappa_t\exp\bp{ 2\kappa_t[X^{\eps v}_3(\tau,\omega)]^n } }
    .
  \eeq
  Furthermore, note that 
    \eqref{eq:X4} 
  implies that for all 
    $t\in[\tau,T]$,
    $\omega\in\Omega$
  it holds that
  \beq
    X^0_4(t,\omega)=0
    .
  \eeq
    This,
    \eqref{eq:X3},
    and \eqref{eq:bndX4} 
  show that for all 
    $\eps\in[0,\infty)$,
    $t\in[\tau,T]$
  it holds that
  \ba
  \label{eq:distEX4bnd}
    \babs{\EE[X^{\eps v}_4(t)] - \EE[X^0_4(t)]}
    &=
    \babs{\EE[X^{\eps v}_4(t)]}
    \\&\geq 
    \bEE{\eps\exp\bp{\kappa_t[X^{\eps v}_3(\tau)]^n-\eps^2\kappa_t\exp\bp{ 2\kappa_t[X^{\eps v}_3(\tau)]^n } }}
    \\&=
    \bEE{\eps\exp\bp{\kappa_t[X^{\eps v}_3(t)]^n-\eps^2\kappa_t\exp\bp{ 2\kappa_t[X^{\eps v}_3(t)]^n } }}
    .
  \ea
  In the next step we observe that
    the fact that $\{t\in\R\colon f(t)>0\}=(\tau,T)$
  ensures that for all 
    $t\in(\tau,T]$ 
  it holds that
  \beq
  \label{eq:kappapos}
    \kappa_t>0.
  \eeq
  In addition, note that
    \eqref{eq:X3}
    and \cref{lem:stdnorm}
    (with
      $\tau\is\tau$,
      $(\Omega,\mc F,\PP)\is(\Omega,\mc F,\PP)$,
      $g\is g$,
      $W\is W|_{[0,\tau]\times\Omega}$,
      $X\is (\Omega\ni\omega\mapsto X_3^{\eps v}(t,\omega)\in\R)$
      for
      $\eps\in[0,\infty)$,
      $t\in[\tau,T]$
    in the notation of \cref{lem:stdnorm})
  demonstrate that for all
    $\eps\in[0,\infty)$,
    $t\in[\tau,T]$
  it holds that
    $X_3^{\eps v}(t)$ is a standard normal random variable.
  Combining 
    this 
    and \eqref{eq:kappapos} 
  with 
    \cref{lem:stricthelp} 
      (with
        $p\is n$,
        $\kappa\is\kappa_t$,
        $c\is n(\kappa_t)^{-2/n}+(\sqrt{2\pi}n+1)\kappa_t+1$,
        $\eps\is\eps$,
        $(\Omega,\mc F,\PP) \is (\Omega,\mc F,\PP)$,
        $Z\is (\Omega\ni\omega\mapsto X^{\eps v}_3(t,\omega)\in\R)$
        for
        $\eps\in(0,\nicefrac1e]$, 
        $t\in(\tau,T)$
      in the notation of \cref{lem:stricthelp})
  proves that for all 
    $t\in(\tau,T)$
  there exist 
    $c\in(0,\infty)$ 
  such that for all 
    $\eps\in(0,\nicefrac1e]$
  it holds that
  \beq
    \bEE{\eps\exp\bp{\kappa_t[X^{\eps v}_3(t)]^n-\eps^2\kappa_t\exp\bp{ 2\kappa_t[X^{\eps v}_3(t)]^n } }}
    \geq 
    \exp\bp{ -c\,\abs{\ln(\eps)}^{2/n} }
    .
  \eeq
    This 
    and \eqref{eq:distEX4bnd} 
  establish \eqref{eq:failhoelderconc}.
  The proof of \cref{prop:failhoelder} is thus completed.
\end{proof}

\section{On the existence of solutions to axis-aligned  SDEs with non-locally H\"older continuous
dependence on the initial values}

\label{sec:axis-aligned-existence}

In this section we establish in \cref{prop:lem1} below, roughly speaking, 
the existence of solutions to certain additive noise driven SDEs whose 
solutions depend 
non-locally H\"older continuously on their initial values and whose
drift coefficient functions are smooth with at most polynomially 
growing derivatives.
In our proof of \cref{prop:lem1} we employ 
\cref{prop:failhoelder} above,
the elementary fact in \cref{lem:exist0} below that certain drift coefficient
functions have at most polynomially growing derivatives,
the elementary fact in \cref{lem:exist} below that appropriate drift coefficient functions 
satisfy a suitable Lyapunov-type condition,
as well as the well-known result 
on the existence of certain smooth bump functions in \cref{lem:bump}
below.
In our proof of \cref{lem:exist0} below we employ the well-known fact
in \cref{lem:opnorm} below
that the Frobenius norm is an upper bound for the operator norm
induced by the standard norm.

\subsection{On drift coefficient functions with at most polynomially growing derivatives}

\begin{lemma}
\label{lem:opnorm}
  Let $d\in\N$, 
  $A=(a_{i,j})_{i,j\in\{1,2,\dots,d\}}\in\R^{d\times d}$
  and let $\norm{\cdot}\colon\R^d\to[0,\infty)$ be the standard norm on $\R^d$.
  Then it holds for all $x\in\R^d$ that
  \beq
    \label{eq:opnorm}
    \norm{Ax}\leq \bbbbbr{\sum_{i,j=1}^d \abs{a_{i,j}}^2}^{\frac12}\norm{x}.
  \eeq
\end{lemma}

\begin{lemma}
  \label{lem:exist0}
  Let $n\in\N\cap[2,\infty)$,
    $c\in[0,\infty)$,
  let $\norm{\cdot}\colon\R^5\to[0,\infty)$ be the standard norm on $\R^5$,
  let $\mu\colon\R^5\to\R^5$,
    $f\in C^1(\R,[0,\infty))$,
    $g\in C^2(\R,[0,\infty))$
  satisfy for all
  $x=(x_1,x_2,x_3,x_4,\allowbreak x_5)\in\R^5$ that
  \beq
  \label{eq:exist0.defmu}
    \mu(x)
    =
    \bp{ 1,0,g'(x_1)x_2,f(x_1)x_4x_5,f(x_1)\bbr{(x_3)^n-(x_4)^2} },
  \eeq
  and assume 
    $c\geq \sup\nolimits_{t\in\R}\bbr{\max\{\abs{f(t)},\abs{f'(t)},\abs{g'(t)},\abs{g''(t)}\}}$.
  Then it holds for all
    $x,h\in\R^5$
  that
    $\mu\in C^1(\R^5,\R^5)$
  and
  \begin{equation}
    \norm{\mu'(x)h}
    \leq 
    4nc(1+\norm{x}^n)\norm h
    .
  \end{equation}
\end{lemma}
\begin{proof}[Proof of \cref{lem:exist0}]
  Observe that
    \eqref{eq:exist0.defmu},
    the hypothesis that
      $f\in C^1(\R,[0,\infty))$,
    and the hypothesis that
      $g\in C^2(\R,[0,\infty))$
  imply that
    for all 
    $x=(x_1,x_2,x_3,x_4,x_5)\in\R^5$ 
  it holds that
    $\mu\in C^1(\R^5,\R^5)$
  and
  \ba
    \label{eq:exist0.diffmu}
    \mu'(x)
    =
    \begin{pmatrix}
      0&0&0&0&0\\
      0&0&0&0&0\\
      g''(x_1)x_2&g'(x_1)&0&0&0\\
      f'(x_1)x_4x_5&0&0&f(x_1)x_5&f(x_1)x_4\\
      f'(x_1)\bbr{(x_3)^n-(x_4)^2}&0&nf(x_1)(x_3)^{n-1}&-2f(x_1)x_4&0
    \end{pmatrix}\!
    .
  \ea
  Moreover, note that
    the hypothesis that $n\geq 2$ 
    and the triangle inequality
  show that for all
    $x=(x_1,x_2,x_3,x_4,x_5)\in\R^5$
  it holds that
  \beq
  \label{eq:diffmu3}
  \begin{split}
    \babs{(x_3)^n-(x_4)^2}^2
    &=
    (x_3)^{2n}-2(x_3)^n(x_4)^2+(x_4)^4
    \leq
    \abs{x_3}^{2n}+2\abs{x_3}^n\abs{x_4}^2+\abs{x_4}^4
    \\&\leq
    \norm{x}^{2n}+2\norm x^{n+2}+\norm{x}^4
    \leq
    4(1+\norm{x}^{2n})
    .
  \end{split}
  \eeq
  Combining 
    this 
  with
    \eqref{eq:exist0.diffmu}
    and \cref{lem:opnorm}
  shows that
  for all
    $x=(x_1,x_2,x_3,x_4,x_5)\in\R^5$,
    $h\in\R^5$
  it holds that
  \begin{align}
  \label{eq:mudiffbnd2}
    \nonumber\norm{\mu'(x)h}
    &\leq
    \bigl[\abs{g''(x_1)}^2\abs{x_2}^2+\abs{g'(x_1)}^2+\abs{f'(x_1)}^2\abs{x_4}^2\abs{x_5}^2+\abs{f(x_1)}^2\abs{x_5}^2
      +\abs{f(x_1)}^2\abs{x_4}^2
      \\\nonumber&\qquad +\abs{f'(x_1)}^2\babs{(x_3)^n-(x_4)^2}^2+n^2\abs{f(x_1)}^2\abs{x_3}^{2n-2}
      +4\abs{f(x_1)}^2\abs{x_4}^2\bigr]^{1/2}\norm h
    \\&\leq \nonumber
    \bigl[c^2\norm{x}^2+c^2+c^2\norm{x}^2\norm{x}^2+c^2\norm{x}^2
      +c^2\norm{x}^2+4c^2(1+\norm{x}^{2n})
      \\&\qquad+n^2c^2\norm{x}^{2n-2}
      +4c^2\norm{x}^2\bigr]^{1/2}\norm h
    \\&=\nonumber
    c\bigl[7\norm{x}^2+1+\norm{x}^4
      +4(1+\norm{x}^{2n})
      +n^2\norm{x}^{2n-2}
    \bigr]^{1/2}\norm h
    \\&\leq\nonumber
    c\bigl[(13+n^2)(1+\norm x^{2n})\bigr]^{1/2}\norm h
    \\&\leq \nonumber
    c\bigl[14n^2(1+\norm x^{2n})\bigr]^{1/2}\norm h
    \leq 
    4cn(1+\norm x^n)\norm h
    .
  \end{align}
  This completes the proof of \cref{lem:exist0}.
\end{proof}

\subsection{On suitable Lyapunov-type functions for additive noise driven SDEs}

\begin{lemma}\label{lem:exist}
  Let
  $n\in\N$,
  $p\in[1,\infty)$,
  $q\in[2pn,\infty)$,
  $f\in C^1(\R,[0,\infty))$,
  $g\in C^2(\R,[0,\infty))$,
  $\sigma=(0,1,0,0,0)\in\R^5$,
  let $\norm{\cdot}\colon\R^5\to[0,\infty)$ be the standard norm on $\R^5$,
  let $\mu\colon\R^5\to\R^5$ satisfy for all
  $x=(x_1,x_2,x_3,x_4,x_5)\in\R^5$ that
  \beq
    \mu(x)=\bp{ 1,0,g'(x_1)x_2,f(x_1)x_4x_5,f(x_1)\bbr{(x_3)^n-(x_4)^2} },
  \eeq
  and let $V\colon\R^5\to[0,\infty)$ 
    \intrtype{be the function which satisfies }%
    \intrtypen{satisfy }%
    for all
      $x=(x_1,x_2,x_3,x_4,x_5)\in\R^5$ 
    that
    \beq
    \label{eq:exist.defV}
      V(x)=\babs{ 1+(x_1)^2+(x_4)^2+(x_5)^2 }^p + \abs{x_2}^q + \abs{x_3}^q + 1.
    \eeq
  Then
    it holds for all
      $x,h\in\R^5$, 
      $z\in\R$
    that
    $V\in C^1(\R^5,[0,\infty))$, $\norm x\leq V(x)$, and
    \ba
    \label{eq:lemexistconc}
      V'(x)\mu(x+\sigma z)
      \leq 
      \bp{2p+(2p+q)\bp{\sup\nolimits_{t\in\R}\bbr{\max\{\abs{f(t)},\abs{g'(t)}\}}}}(1+\abs{z})V(x)
      .
    \ea
\end{lemma}
\begin{proof}[Proof of \cref{lem:exist}]
  Throughout this proof
  let $c\in[0,\infty]$ satisfy
  \beq
  \label{eq:defc}
    c=\sup\nolimits_{t\in\R}\bbr{\max\{\abs{f(t)},\abs{g'(t)}\}}
  \eeq
  and assume w.l.o.g.\ that $c<\infty$.
  Observe that 
    the hypothesis that $p\geq 1$
    and the hypothesis that $q\geq 2pn\geq 2$
  show that for all 
    $x=(x_1,x_2,x_3,x_4,x_5)\in\R^5$ 
  it holds that
  \ba
    \norm{x}^2
    &=
    \abs{x_1}^2+\abs{x_2}^2+\abs{x_3}^2+\abs{x_4}^2+\abs{x_5}^2\\
    &\leq 
    \abs{x_1}^2+\bp{1+\abs{x_2}^q}+\bp{1+\abs{x_3}^q}+\abs{x_4}^2+\abs{x_5}^2\\
    &\leq 
    \bp{ 1+\abs{x_1}^2+\abs{x_4}^2+\abs{x_5}^2}^p+\abs{x_2}^q+\abs{x_3}^q+1
    =
    V(x)
    .
  \ea
    The fact that for all $x\in\R^5$ it holds that $V(x)\geq 1$
    hence
  ensures that
  \beq
  \label{eq:Vbnd}
    \norm x\leq [V(x)]^{1/2}
    \leq 
    V(x)
    .
  \eeq
  Furthermore, note that
    \eqref{eq:exist.defV}, 
    the triangle inequality,
    the hypothesis that
      $q\geq 2pn\geq 2$,
    and the fact that 
      for all
        $r\in(1,\infty)$,
        $f\in C(\R,\R)$,
        $x\in\R$
        with $f=(\R\ni y\mapsto \abs{y}^r\in\R)$
      it holds that
        $\abs{f'(x)}=r\abs{x}^{r-1}$
  imply that for all
    $x=(x_1,x_2,x_3,x_4,x_5)$,
    $v=(v_1,v_2,v_3,v_4,v_5)\in\R^5$
  it holds that
    $V\in C^1(\R^5,[0,\infty))$
  and
  \bm
    \abs{V'(x)v}
    \leq
    \babs{2p\bp{ 1+(x_1)^2+(x_4)^2+(x_5)^2 }^{p-1} (x_1v_1+x_4v_4+x_5v_5)}
      \\
      + q\abs{x_2}^{q-1}\abs{v_2} + q\abs{x_3}^{q-1}\abs{v_3}
    .
  \em
  This proves that for all
    $x=(x_1,x_2,x_3,x_4,x_5)\in\R^5$,
    $z\in\R$
  it holds that
  \ba
  \label{eq:Vdiffbnd}
    &\abs{V'(x)\mu(x+\sigma z)}\\
    &=
    \abs{V'(x)\mu(x_1,x_2+z,x_3,x_4,x_5)}\\
    &= 
    \babs{V'(x)\bp{1,0,g'(x_1)(x_2+z),f(x_1)x_4x_5,f(x_1)((x_3)^n-(x_4)^2) }}\\
    &\leq
    \babs{2p\bp{ 1+(x_1)^2+(x_4)^2+(x_5)^2 }^{p-1}\bp{ x_1+f(x_1)(x_4)^2x_5+f(x_1)((x_3)^n-(x_4)^2)x_5 } } \\
      &\qquad + q\abs{x_3}^{q-1}\abs{g'(x_1)(x_2+z)}\\
    &\leq
    2p\babs{ 1+(x_1)^2+(x_4)^2+(x_5)^2 }^{p-1} \bp{\abs{x_1}+\abs{f(x_1)}\abs{x_3}^n\abs{x_5}}
    \\&\qquad
      + q\abs{x_3}^{q-1}\abs{g'(x_1)}(\abs{x_2}+\abs{z})
    .
  \ea
  Next observe that 
    \eqref{eq:exist.defV}
  ensures that for all
    $x=(x_1,x_2,x_3,x_4,x_5)\in\R^5$
  it holds that
  \beq
  \label{eq:Vdiffbnd1}
    \abs{x_1}\babs{ 1+(x_1)^2+(x_4)^2+(x_5)^2 }^{p-1}
    \leq 
    \babs{ 1+(x_1)^2+(x_4)^2+(x_5)^2 }^p\leq V(x)
    .
  \eeq
  In addition, note that
  \eqref{eq:exist.defV},
  \eqref{eq:defc},
    the fact that $\frac1{2p}+\frac{p-(1/2)}{p}=1$,
    the Young inequality, 
    and the hypothesis that $q\geq 2pn$
  demonstrate that for all
    $x=(x_1,x_2,x_3,x_4,x_5)\in\R^5$
  it holds that
  \ba
  \label{eq:Vdiffbnd2}
    &\abs{f(x_1)} \abs{x_3}^n\abs{x_5}\babs{1+(x_1)^2+(x_4)^2+(x_5)^2}^{p-1}\\
    &\leq 
    c\,\abs{x_3}^n \babs{(x_5)^2}^{1/2}\babs{1+(x_1)^2+(x_4)^2+(x_5)^2}^{p-1}\\
    &\leq 
    c\,\abs{x_3}^n \babs{1+(x_1)^2+(x_4)^2+(x_5)^2}^{p-1/2}\\
    &\leq 
    c\bbbbbr{ \frac{\abs{x_3}^{2pn}}{2p} + \frac{\babs{ 1+(x_1)^2+(x_4)^2+(x_5)^2}^{p}}{\frac{p}{p-\frac12}} }\\
    &\leq 
    c \bbr{ 1+\abs{x_3}^{q} + \babs{ 1+(x_1)^2+(x_4)^2+(x_5)^2 }^{p} }
    \leq 
    c V(x)
    .
  \ea
  Next observe that
    \eqref{eq:exist.defV},
    the fact that $\frac1q+\frac{q-1}q=1$,
    and the Young inequality
  show that for all
    $x=(x_1,x_2,x_3,x_4,x_5)\in\R^5$
  it holds that
  \ba
  \label{eq:Vdiffbnd3}
    \abs{x_3}^{q-1} \abs{g'(x_1)} \abs{x_2}
    &\leq 
    c\,\abs{x_2}\abs{x_3}^{q-1}
    \leq 
    c \bbbbr{ \frac{\abs{x_2}^q}q  + \frac{(\abs{x_3}^{q-1})^{\frac q{q-1}}}{\frac q{q-1}}}
    \\&\leq 
    c\bbr{\abs{x_2}^q+\abs{x_3}^q}
    \leq 
    cV(x)
    .
  \ea
  Moreover, note that 
    the fact that $\frac1q+\frac{q-1}q=1$
  implies that for all
    $x=(x_1,x_2,x_3,x_4,x_5)\in\R^5$,
    $z\in\R$
  it holds that
  \beq
    \abs{x_3}^{q-1} \abs{g'(x_1)} \abs{z}
    \leq 
    c\,\abs{z}\abs{x_3}^{q-1}
    \leq 
    c\,\abs z\bp{1+\abs{x_3}^q}
    \leq 
    c\,\abs zV(x)
    .
  \eeq
  Combining
    this, 
    \eqref{eq:Vdiffbnd1}, 
    \eqref{eq:Vdiffbnd2}, 
  and 
    \eqref{eq:Vdiffbnd3}
  with
    \eqref{eq:Vdiffbnd}
  proves that
  \ba
  \label{eq:Vdiffreq}
    V'(x)\mu(x+\sigma z)
    &\leq
    \abs{V'(x)\mu(x+\sigma z)}
    \leq 
    2p V(x)+2pcV(x)+qcV(x)+qc\,\abs{z} V(x)
    \\&=
    (2p+2pc)V(x)+qc(1+\abs{z})V(x)
    \leq 
    (2p+2pc+qc)(1+\abs{z})V(x)
    .
  \ea
    This
    and \eqref{eq:Vbnd}
  establish
    \eqref{eq:lemexistconc}.
  The proof of \cref{lem:exist} is thus completed.
\end{proof}

\subsection{On solutions to SDEs with irregularities in the initial value}

\begin{lemma}
\label{lem:bump}
  Let $a\in\R$, $b\in(a,\infty)$.
  Then there exists a function $f\in C^{\infty}(\R,[0,\infty))$ which satisfies that
    $\{t\in\R\colon f(t)> 0\}=(a,b)$
    and $\int_a^b \abs{f(t)}^2\,\diff t=1$.
\end{lemma}

\begin{prop}
\label{prop:lem1}
  Let
  $d\in\{5,6,\dots\}$,
  $n\in\{2,3,\dots\}$,
  $T\in(0,\infty)$,
  $\tau\in(0,T)$,
  let $\norm\cdot\colon\R^d\to[0,\infty)$ be the standard norm on $\R^d$,
  let $(\Omega,\mc F,\PP)$ be a probability space,
  and let $W\colon[0,T]\times\Omega\to\R$ be a standard Brownian motion
    with continuous sample paths.
  Then there exist 
    $\mu\in C^\infty(\R^d,\R^d)$,
    $\sigma\in\R^d$,
    $V\in C^{\infty}(\R^d,[0,\infty))$,
    $\kappa\in[1,\infty)$
  such that
  \begin{enumerate}[label=(\roman{enumi}),ref=(\roman{enumi})]
    \item \label{it:lem1.1}
    it holds for all
      $x,h\in\R^d$, 
      $z\in\R$ 
    that
      $\norm{\mu'(x)h}\leq \kappa \bp{1+\norm x^\kappa} \norm h$,
      $V'(x)\mu(x+\sigma z)\leq \kappa(1+\abs{z})V(x)$,
      and $\norm x\leq V(x)$,
    \item \label{it:lem1.2}
      there exist unique stochastic processes
      $X^x\colon [0,T]\times\Omega\to\R^d$, $x\in\R^d$, with continuous sample paths
      such that for all 
        $x\in\R^d$,
        $t\in[0,T]$,
        $\omega\in\Omega$
      it holds that
      \beq
        X^x(t,\omega)=x+\int_0^t\mu(X^x(s,\omega))\,\diff s+\sigma W(t,\omega)
        ,
      \eeq
      and
    \item \label{it:lem1.3}
      it holds for all 
        $t\in(\tau,T)$ 
      that 
        there exists
          $c\in(0,\infty)$ 
        such that
          for all
            $\eps\in(0,\nicefrac1e]$,
            $h=(0,0,0,\eps,0,0,\dots,0)\in\R^d$,
          it holds that
          \beq
          \label{eq:lem1nonhoelder}
            \exp\bp{ -c\,\abs{\ln(\norm{h})}^{2/n} }
            =
            \exp\bp{ -c\,\abs{\ln(\eps)}^{2/n} }
            \leq 
            \bEE{ \norm{X^h(t)-X^0(t)} } 
            .
          \eeq
  \end{enumerate}
\end{prop}
\begin{proof}[Proof of \cref{prop:lem1}]
  \newcommand{\vproj}{\varpi}
  \newcommand{\vprojj}{\mathbf p}
  Throughout this proof
    let $\nnorm\cdot\colon\R^5\to[0,\infty)$
      be the standard norm on $\R^5$,
    let $f,g\in C^\infty(\R,[0,\infty))$
      satisfy
        $\{t\in\R\colon f(t)> 0\}=(\tau,T)$,
        $\{t\in\R\colon g(t)> 0\}=(0,\tau)$,
        and $\int_0^\tau\abs{g(t)}^2=1$,
  let $\rho=(0,1,0,0,0)\in\R^5$,
  let $\sigma=(0,1,0,0,\dots,0)\in\R^d$,
  let $C,\kappa\in[0,\infty)$ satisfy
    \begin{equation}
      C=\sup\nolimits_{t\in\R}\bbr{\max\{1,\abs{f(t)},\abs{f'(t)},\abs{g'(t)},\abs{g''(t)}\}}
      \qqandqq
      \kappa=2+8(n+1)C
      ,
    \end{equation}
  let $\vproj\colon\R^d\to\R^5$ satisfy
    for all
      $x=(x_1,x_2,\dots,x_d)\in\R^d$
    that
      $\vproj(x)=(x_1,x_2,x_3,x_4,x_5)$,
  let $U\colon\R^5\to[0,\infty)$ 
    \intrtype{be the function which satisfies }%
    \intrtypen{satisfy }%
    for all
      $x=(x_1,x_2,x_3,x_4,x_5)\in\R^5$ 
    that
    \beq
    \label{eq:lem1.defU}
      U(x)= 1+(x_1)^2+(x_4)^2+(x_5)^2 + (x_2)^{2n} + (x_3)^{2n} + 1
      ,
    \eeq
  let $V\colon\R^d\to[0,\infty)$ 
    satisfy for all
      $x\in\R^d$
    that
    \beq
      \label{eq:lem1.defV}
      V(x)
      =
      U(\varpi(x))+\bbbbbr{\sum_{i\in\N\cap(5,d+1)} (x_i)^2}+1
      ,
    \eeq
  let $\nu\colon\R^5\to\R^5$ satisfy for all
    $x=(x_1,x_2,x_3,x_4,x_5)\in\R^5$
  that
  \beq
  \label{eq:lem1.defnu}
    \nu(x)=\bp{ 1,0,g'(x_1)x_2,f(x_1)x_4x_5,f(x_1)\bbr{(x_3)^n-(x_4)^2} },
  \eeq
  and let $\mu\colon\R^d\to\R^d$ 
    satisfy for all
      $x,\,y=(y_1,y_2,\dots,y_d)\in\R^d$
    with $\mu(x)=y$
    that
    \beq
    \label{eq:lem1.defmu}
      \vproj(y)=\nu(\vproj(x))
      \qquad\text{and}\qquad
      \forall\,i\in\N\cap(5,d+1)\colon y_i=0
    \eeq
  (cf.\ \cref{lem:bump}).
  Observe that
    \eqref{eq:lem1.defnu},
    the fact that $f,g\in C^{\infty}(\R,[0,\infty))$,
    and \cref{lem:exist0}
    (with 
      $n\is n$,
      $\norm\cdot\is\nnorm\cdot$,
      $\mu\is\nu$,
      $f\is f$,
      $g\is g$,
      $c\is C$
    in the notation of \cref{lem:exist0})
  establish that for all
    $x,h\in\R^5$
  it holds that
    $\nu\in C^\infty(\R^5,\R^5)$
  and
  \begin{equation}
    \begin{split}
    \nnorm{\nu'(x)h}
    &\leq
    4nC(1+\nnorm x^n)\,\nnorm h
    \leq
    4nC(2+\nnorm x^\kappa)\,\nnorm h
    \\&\leq
    8nC(1+\nnorm x^\kappa)\,\nnorm h
    \leq
    \kappa(1+\nnorm x^\kappa)\,\nnorm h
    .
    \end{split}
  \end{equation}
    This
    and \eqref{eq:lem1.defmu}
  show that for all 
    $x,h\in\R^d$
  it holds that
    $\mu\in C^\infty(\R^d,\R^d)$
    and 
  \begin{equation}
    \label{eq:lem1.mu}
    \begin{split}
    \norm{\mu'(x)h}
    &=
    \nnorm*{\bbr{\nu'(\vproj(x))}\vproj(h)}
    \leq
    \kappa(1+\nnorm{\vproj(x)}^\kappa)\,\nnorm{\vproj(h)}
    \leq
    \kappa(1+\norm{x}^\kappa)\,\norm{h}
    .
    \end{split}
  \end{equation}
  In the next step we note that
    \eqref{eq:lem1.defU},
    the fact that $f,g\in C^{\infty}(\R,[0,\infty))$,
    and \cref{lem:exist}
    (with
      $n\is n$,
      $p\is 1$,
      $q\is 2n$,
      $f\is f$,
      $g\is g$,
      $\sigma\is\rho$,
      $\norm\cdot\is\nnorm\cdot$,
      $\mu\is\nu$,
      $V\is U$
    in the notation of \cref{lem:exist})
  prove that for all
    $x,h\in\R^5$,
    $z\in\R$
  it holds that
  \begin{equation}
    \label{eq:lem1.Uprops}
    U\in C^\infty(\R^5,[0,\infty)),\;\,
    \nnorm x\leq U(x),
    \,\;\text{and}\;\,
    U'(x)\nu(x+\rho z)
    \leq
    2(1+C+nC)(1+\abs z)U(x)
    .
  \end{equation}
  Combining
    this
  with
    \eqref{eq:lem1.defV}
  demonstrates for all
    $x=(x_1,x_2,\dots,x_d)\in\R^d$
  that
  \begin{equation}
    \label{eq:lem1.V1}
    \begin{split}
    \norm{x}
    &=
    \bbbbbr{
      \nnorm{\vproj(x)}^2
      +
      \bbbbp{\sum_{i\in\N\cap(5,d+1)} (x_i)^2}
    }^{\frac12}
    \leq
    \nnorm{\vproj(x)}+\bbbbbr{\sum_{i\in\N\cap(5,d+1)} (x_i)^2}^{\frac12}
    \\&\leq
    U(\vproj(x))+\max\Biggl\{1,\bbbbbr{\sum_{i\in\N\cap(5,d+1)} (x_i)^2}\Biggr\}
    \leq
    V(x)
    .
    \end{split}
  \end{equation}
  Moreover, observe that
    \eqref{eq:lem1.defV}
    and \eqref{eq:lem1.Uprops}
  imply that for all
    $x=(x_1,x_2,\dots,x_d),\,\allowbreak h=(h_1,h_2,\dots,\allowbreak h_d)\in\R^d$
  it holds that $V\in C^{\infty}(\R^d,[0,\infty))$ and
  \begin{equation}
    V'(x)h
    =
    \bbr{U'(\vproj(x))}\vproj(h)+\bbbbbr{\sum_{i\in\N\cap(5,d+1)}(2x_ih_i)}
    .
  \end{equation}
    This,
    \eqref{eq:lem1.defmu},
    and \eqref{eq:lem1.Uprops}
  ensure that for all
    $x\in\R^d$,
    $z\in\R$
  it holds that
  \begin{equation}
    \label{eq:lem1.Vprop}
    \begin{split}
      &V'(x)\mu(x+\sigma z)
      =
      U'(\varpi(x))\vproj(\mu(x+\sigma z))
      =
      U'(\vproj(x))\nu\bp{\vproj(x)+\rho z}
      \\&\leq
      2(1+C+nC)(1+\abs z)U(\vproj(x))
      \leq
      2(1+C+nC)(1+\abs z)V(x)
      \\&\leq
      \kappa(1+\abs z)V(x)
      .
    \end{split}
  \end{equation}
    This,
    \eqref{eq:lem1.mu},
    and \eqref{eq:lem1.V1}
  establish \cref{it:lem1.1}.
  Next we combine
    \eqref{eq:lem1.Uprops}
    and \cite[Lemma~5.4]{JentzenKuckuckMuellerGronbachYaroslavtseva2019}
    (with
      $d\is 5$,
      $m\is 1$,
      $T\is T$,
      $\mu\is\nu$,
      $\sigma\is\rho$,
      $\varphi\is (\R\ni z\mapsto 2(1+C+nC)(1+\abs z)\in[0,\infty))$,
      $V\is U$,
      $\norm\cdot\is\nnorm\cdot$,
      $(\Omega,\mc F,\PP)\is(\Omega,\mc F,\PP)$,
      $W\is W$
    in the notation of \cite[Lemma~5.4]{JentzenKuckuckMuellerGronbachYaroslavtseva2019})
  to obtain that there exist unique stochastic processes 
    $Y^x\colon[0,T]\times\Omega\to\R^5$, $x\in\R^5$,
    with continuous sample paths
  which satisfy for all
    $x\in\R^5$,
    $t\in[0,T]$,
    $\omega\in\Omega$
  that
  \beq
  \label{eq:lem1.defY}
    Y^x(t,\omega)=x+\int_0^t\nu(Y^x(s,\omega))\,\diff s+\rho W(t,\omega)
    .
  \eeq
  In addition, note that
    \eqref{eq:lem1.V1},
    \eqref{eq:lem1.Vprop},
    and \cite[Lemma~5.4]{JentzenKuckuckMuellerGronbachYaroslavtseva2019}
    (with
      $d\is d$,
      $m\is 1$,
      $T\is T$,
      $\mu\is\mu$,
      $\sigma\is\sigma$,
      $\varphi\is (\R\ni z\mapsto \kappa(1+\abs z)\in[0,\infty))$,
      $V\is V$,
      $\norm\cdot\is\norm\cdot$,
      $(\Omega,\mc F,\PP)\is(\Omega,\mc F,\PP)$,
      $W\is W$
    in the notation of \cite[Lemma~5.4]{JentzenKuckuckMuellerGronbachYaroslavtseva2019})
  ensure that there exist unique stochastic processes 
    $X^x\colon[0,T]\times\Omega\to\R^d$, $x\in\R^d$,
    with continuous sample paths
  which satisfy for all
    $x\in\R^d$,
    $t\in[0,T]$,
    $\omega\in\Omega$
  that
  \beq
  \label{eq:lem1.defX}
    X^x(t,\omega)=x+\int_0^t\mu(X^x(s,\omega))\,\diff s+\sigma W(t,\omega)
    .
  \eeq
  This proves \cref{it:lem1.2}.
  In the next step let $Z^x\colon [0,T]\times\Omega\to\R^d$, $x\in\R^d$, satisfy for all
    $x=(x_1,x_2,\dots,x_d),\,y=(y_1,y_2,\dots,y_d)\in\R^d$,
    $t\in[0,T]$,
    $\omega\in\Omega$
    with $Z^x(t,\omega)=y$
  that
  \beq
  \label{eq:lem1.defZ}
    \vproj(y)=Y^{\vproj(x)}(t,\omega)
    \qqandqq
    \forall\, i\in\N\cap(5,d+1)\colon y_i=x_i
    .
  \eeq
  Observe that
    \eqref{eq:lem1.defmu}
    and \eqref{eq:lem1.defY}
  demonstrate that for all
    $x\in\R^d$,
    $t\in[0,T]$,
    $\omega\in\Omega$
  it holds that
  \begin{equation}
    \begin{split}
    \vproj(Z^x(t,\omega))
    &=
    Y^{\vproj(x)}(t,\omega)
    =
    \vproj(x)+\int_0^t \nu(Y^{\vproj(x)}(s,\omega))\,\diff s+\rho W(t,\omega)
    \\&=
    \vproj(x)+\int_0^t \nu\bp{\vproj(Z^{x}(s,\omega))}\,\diff s+\rho W(t,\omega)
    \\&=
    \vproj(x)+\int_0^t \vproj\bp{\mu(Z^{x}(s,\omega))}\,\diff s+\rho W(t,\omega)
    \\&=
    \vproj(x)+\vproj\bbbp{\int_0^t \mu(Z^{x}(s,\omega))\,\diff s}+\vproj(\sigma) W(t,\omega)
    \\&=
    \vproj\bbbp{x+\int_0^t \mu(Z^{x}(s,\omega))\,\diff s+\sigma W(t,\omega)}
    .
    \end{split}
  \end{equation}
    This
    and the fact that 
      for all
        $x\in\R^d$,
        $t\in[0,T]$,
        $\omega\in\Omega$,
        $y,z\in\R^d$,
        $i\in\N\cap(5,d+1)$
        with
          $y=\int_0^t\mu(Z^{x}(s,\omega))\,\diff s$
          and $z=\rho W(t,\omega)$
      it holds that $y_i=0=z_i$
  establishes that for all
    $x\in\R^d$,
    $t\in[0,T]$,
    $\omega\in\Omega$
  it holds that
  \begin{equation}
    Z^x(t,\omega)
    =
    x+\int_0^t\mu(Z^x(s,\omega))\,\diff s+\sigma W(t,\omega)
    .
  \end{equation}
  Combining
    this
  with
    \eqref{eq:lem1.defX}
  shows that for all
    $x\in\R^d$,
    $t\in[0,T]$,
    $\omega\in\Omega$
  it holds that
  \begin{equation}
    \label{eq:lem1.eqXZ}
    X^x(t,\omega)
    =
    Z^x(t,\omega)
    .
  \end{equation}
  Next note that 
    \eqref{eq:lem1.defY}
    and \cref{prop:failhoelder}
    (with
      $T\is T$,
      $\tau\is\tau$,
      $n\is n$,
      $f\is f$,
      $g\is g$,
      $\sigma\is\rho$,
      $\norm\cdot\is\nnorm\cdot$,
      $\mu\is\nu$,
      $(\Omega,\mc F,\PP)\is(\Omega,\mc F,\PP)$,
      $W\is W$,
      $(X^x)_{x\in\R^5}\is(Y^x)_{x\in\R^5}$
    in the notation of \cref{prop:failhoelder})
  prove that for all
    $t\in(\tau,T)$
  there exists
    $c\in(0,\infty)$
  such that for all
    $\eps\in(0,\nicefrac1e]$,
    $h=(0,0,0,\eps,0)\in\R^5$
  it holds that
  \ba
  \label{eq:lem1.failhoelder5}
    \exp(-c\,\abs{\ln(\eps)}^{2/n})
    &\leq
    \bEE{\nnorm{Y^h(t)-Y^0(t)}}
    .
  \ea
  Note that
    \eqref{eq:lem1.defZ}
  implies that for all
    $x=(x_1,x_2,\dots,x_d),\,y=(y_1,y_2,\dots,y_d)\in\R^d$,
    $t\in[0,T]$,
    $\omega\in\Omega$
  it holds that
  \begin{equation}
    \begin{split}
    &\norm{Z^x(t,\omega)-Z^y(t,\omega)}
    \\&=
    \bbbbbr{
      \nnorm{
      Y^{\vproj(x)}(t,\omega)
      -
      Y^{\vproj(y)}(t,\omega)
      }^2
      +
      \bbbbp{
      \sum_{i\in\N\cap(5,d+1)}
        (x_i-y_i)^2
      }
    }^{\frac12}
    \\&\geq
    \nnorm{
    Y^{\vproj(x)}(t,\omega)
    -
    Y^{\vproj(y)}(t,\omega)
    }
    .
    \end{split}
  \end{equation}
    This,
    \eqref{eq:lem1.eqXZ}
    and \eqref{eq:lem1.failhoelder5}
  demonstrate that for all
    $t\in(\tau,T)$
  there exists
    $c\in(0,\infty)$
  such that for all
    $\eps\in(0,\nicefrac1e]$,
    $h=(0,0,0,\eps,0,0,\dots,0)\in\R^d$
  it holds that
  \begin{equation}
    \begin{split}
    &\exp\bp{-c\,\abs{\ln(\eps)}^{2/n}}
    \leq
    \bEE{\norm{Y^{\vproj(h)}(t)-X^0(t)}}
    \\&\leq
    \bEE{\norm{Z^h(t)-Z^0(t)}}
    =
    \bEE{\norm{X^h(t)-X^0(t)}}
    .
    \end{split}
  \end{equation}
  This establishes \cref{it:lem1.3}.
  The proof of \cref{prop:lem1} is thus completed.
\end{proof}

\section{On the existence of solutions to SDEs with non-locally H\"older continuous
dependence on the initial values}

\label{sec:general}

In this section we establish in
\cref{thm:final} below 
the existence of solutions to certain additive noise driven SDEs 
whose solutions depend 
non-locally H\"older continuously on their initial values and whose
drift coefficient functions are smooth with at most polynomially growing 
derivatives.
In our proof of \cref{thm:final} we employ
\cref{lem:lem2} below as well as
the elementary estimate for certain real-valued functions in \cref{lem:hoeldercomp} below.
\cref{lem:lem2} is a strengthened version of \cref{prop:lem1} above.
Our proof of \cref{lem:lem2} employs the essentially well-known
fact
on affine linear transformations of solutions to SDEs in
\cref{lem:transformsolution} below. 
For completeness we also provide in
this section a detailed proof of \cref{lem:transformsolution}.

\subsection{On affine linear transformations of SDEs}

\begin{lemma}
  \label{lem:transformsolution}
  Let $d,m\in\N$,
  $T\in(0,\infty)$,
  $\nu\in C^1(\R^d,\R^d)$,
  $\rho\in\R^{d\times m}$,
  $v\in\R^d$,
  let $(\Omega,\mc F,\PP)$ be a probability space,
  let $W\colon [0,T]\times\Omega\to \R^m$ and
    $Y^x\colon [0,T]\times\Omega\to\R^d$, $x\in\R^d$, 
    be stochastic processes with continuous sample paths,
  assume for all
    $x\in\R^d$,
    $t\in[0,T]$,
    $\omega\in\Omega$
  that
  \begin{equation}
    Y^x(t,\omega)
    =
    x+\int_0^t \nu(Y^x(s,\omega))\,\diff s+\rho W(t,\omega),
  \end{equation}
  let $L\in\R^{d\times d}$ 
    be an invertible matrix,
  let $\mu\colon \R^d\to\R^d$ 
    satisfy for all
      $x\in\R^d$
    that
      $\mu(x)=L\nu(L^{-1}(x-v))$,
  and let $X^x\colon[0,T]\times\Omega\to\R^d$, $x\in\R^d$,
    satisfy for all
      $x\in\R^d$,
      $t\in[0,T]$,
      $\omega\in\Omega$
    that
    \begin{equation}
      \label{eq:transformsolution.defX}
      X^x(t,\omega)
      =
      LY^{L^{-1}(x-v)}(t,\omega)+v
      .
    \end{equation}
  Then it holds for all
    $x\in\R^d$,
    $t\in[0,T]$,
    $\omega\in\Omega$
  that
  \begin{equation}
    X^x(t,\omega)
    =
    x+\int_0^t\mu(X^x(s,\omega))\,\diff s+L\rho W(t,\omega)
    .
  \end{equation}
\end{lemma}
\begin{proof}[Proof of \cref{lem:transformsolution}]
  Note that 
    \eqref{eq:transformsolution.defX}
  ensures that for all
    $x\in\R^d$,
    $s\in[0,T]$,
    $\omega\in\Omega$
  it holds that
  \begin{equation}
    \begin{split}
    \mu(X^x(s,\omega))
    &=
    L\nu(L^{-1}(X^x(s,\omega)-v))
    \\&=
    L\nu\bp{L^{-1}([LY^{L^{-1}(x-v)}(s,\omega)+v]-v)}
    \\&=
    L\nu\bp{L^{-1}LY^{L^{-1}(x-v)}(s,\omega)}
    \\&=
    L\nu\bp{Y^{L^{-1}(x-v)}(s,\omega)}
    .
    \end{split}
  \end{equation}
  Therefore, it holds for all
    $x\in\R^d$,
    $t\in[0,T]$,
    $\omega\in\Omega$
  that
  \begin{equation}
    \begin{split}
    X^x(t,\omega)
    &=
    LY^{L^{-1}(x-v)}(t,\omega)+v
    \\&=
    L\bbbbr{
      L^{-1}(x-v)+\int_0^t\nu(Y^{L^{-1}(x-v)}(s,\omega))\,\diff s+\rho W(t,\omega)
    }+v
    \\&=
    [x-v]+L\bbbbr{\int_0^t\nu(Y^{L^{-1}(x-v)}(s,\omega))\,\diff s}+L\rho W(t,\omega)+v
    \\&=
    x+\int_0^t L\nu(Y^{L^{-1}(x-v)}(s,\omega))\,\diff s+L\rho W(t,\omega)
    \\&=
    x+\int_0^t \mu(X^x(s,\omega))\,\diff s+L\rho W(t,\omega)
    .
    \end{split}
  \end{equation}
  This completes the proof of \cref{lem:transformsolution}.
\end{proof}

\subsection{On solutions to SDEs with irregularities in the initial value}

\begin{lemma}
\label{lem:lem2}
  Let $d\in\{5,6,\dots\}$,
  $n\in\{2,3,\dots\}$,
  $T\in(0,\infty)$,
  $\tau\in(0,T)$,
  $v\in\R^d$,
  $\delta\in\R^d\setminus\{0\}$,
  let $\norm\cdot\colon\R^d\to[0,\infty)$ be the standard norm on $\R^d$,
  let $(\Omega,\mc F,\PP)$ be a probability space,
  and let $W\colon [0,T]\times\Omega\to\R$ be a standard Brownian motion
    with continuous sample paths.
  Then there exist 
    $\mu\in C^\infty(\R^d,\R^d)$,
    $\sigma\in\R^d$,
    $V\in C^\infty(\R^d,[0,\infty))$,
    $\kappa\in(0,\infty)$
  such that
  \begin{enumerate}[label=(\roman{enumi}),ref=(\roman{enumi})]
    \item \label{it:lem2.1}
      it holds for all
        $x,h\in\R^d$, 
        $z\in\R$ 
      that
        $\norm{\mu'(x)h}\leq \kappa \bp{1+\norm x^\kappa} \norm h$,
        $V'(x)\mu(x+\sigma z)\leq \kappa(1+\abs{z})V(x)$,
        and $\norm x\leq V(x)$,
    \item \label{it:lem2.2}
      there exist unique stochastic processes
      $X^x\colon [0,T]\times\Omega\to\R^d$, $x\in\R^d$, with continuous sample paths
      such that for all 
        $x\in\R^d$,
        $t\in[0,T]$,
        $\omega\in\Omega$
      it holds that
      \beq
        X^x(t,\omega)=x+\int_0^t\mu(X^x(s,\omega))\,\diff s+\sigma W(t,\omega)
        ,
      \eeq
      and
    \item \label{it:lem2.3}
      it holds for all 
        $t\in(\tau,T)$ 
      that 
        there exists
          $c\in(0,\infty)$ 
        such that
          for all
            $w\in\{v+r\delta\colon r\in(0,\nicefrac1e]\}$
          it holds that
          \beq
            \norm\delta\exp\bp{ -c\,\abs{\ln(\norm{v-w})}^{2/n} }
            \leq 
            \bEE{ \norm{X^v(t)-X^w(t)} } 
            .
          \eeq
  \end{enumerate}
\end{lemma}
\begin{proof}[Proof of \cref{lem:lem2}]
  Throughout this proof, 
  let $u=(0,0,0,1,0,\dots,0)\in\R^d$,
  let
    $A\in\R^{d\times d}$
  be an orthogonal matrix which satisfies
  \beq
  \label{eq:lem2.defA}
    Au=\frac{\delta}{\norm \delta},
  \eeq
  and let $B\in\R^{d\times d}$ satisfy
  \begin{equation}
    \label{eq:lem2.defB}
    B=\norm\delta A.
  \end{equation}
  Note that 
    \cref{prop:lem1}
    (with 
      $d\is d$,
      $n\is n$,
      $T\is T$,
      $\tau\is\tau$,
      $\norm\cdot\is\norm\cdot$,
      $(\Omega,\mc F,\PP)\is(\Omega,\mc F,\PP)$,
      $W\is W$
    in the notation of \cref{prop:lem1})
  shows that there exist
    $\nu\in C^\infty(\R^d,\R^d)$,
    $\rho\in\R^d$,
    $U\in C^\infty(\R^d,[0,\infty))$,
    $\varkappa\in[1,\infty)$
  which satisfy that
  \begin{enumerate}[label=(\Alph{enumi}),ref=(\Alph{enumi})]
    \item \label{it:lem2.lem1.1}
      it holds for all
        $x,h\in\R^d$, 
        $z\in\R$ 
      that
        $\norm{\nu'(x)h}\leq \varkappa \bp{1+\norm x^\varkappa} \norm h$,
        $U'(x)\nu(x+\rho z)\leq \varkappa(1+\abs{z})U(x)$,
        and $\norm x\leq U(x)$,
    \item
      there exist unique stochastic processes
      $Y^x\colon [0,T]\times\Omega\to\R^d$, $x\in\R^d$, with continuous sample paths
      such that for all 
        $x\in\R^d$,
        $t\in[0,T]$,
        $\omega\in\Omega$
      it holds that
      \beq
        Y^x(t,\omega)=x+\int_0^t\nu(Y^x(s,\omega))\,\diff s+\rho W(t,\omega)
        ,
      \eeq
      and
    \item
    \label{it:lem2.lem1.3}
      it holds for all 
        $t\in(\tau,T)$ 
      that 
        there exists
          $c\in(0,\infty)$ 
        such that
          for all
            $w\in\{ru\colon r\in(0,\nicefrac1e]\}$
          it holds that
          \beq
            \exp\bp{ -c\,\abs{\ln(\norm{w})}^{2/n} }
            \leq 
            \bEE{ \norm{Y^0(t)-Y^w(t)} } 
            .
          \eeq
  \end{enumerate}
  Next let 
    $\mu\in C^\infty(\R^d,\R^d)$,
    $\sigma\in\R^d$,
    and $Z^x\colon [0,T]\times\Omega\to\R^d$, $x\in\R^d$,
    satisfy for all 
      $x\in\R^d$,
      $t\in[0,T]$,
      $\omega\in\Omega$
    that
    \beq
      \label{eq:lem2.defX}
      \mu(x)=B\nu(B^{-1}(x-v)),
      \qquad
      \sigma=B\rho,
      \qqandqq
      Z^x(t,\omega)=BY^{B^{-1}(x-v)}(t,\omega)+v.
    \eeq
  Observe that
    \cref{lem:transformsolution}
    (with
      $d\is d$,
      $m\is 1$,
      $T\is T$,
      $\nu\is\nu$,
      $\rho\is\rho$,
      $v\is v$,
      $(\Omega,\mc F,\PP)\is(\Omega,\mc F,\PP)$,
      $W\is W$,
      $(Y^x)_{x\in\R^d}\is(Y^x)_{x\in\R^d}$,
      $L\is B$,
      $\mu\is\mu$,
      $(X^x)_{x\in\R^d}\is(Z^x)_{x\in\R^d}$
    in the notation of \cref{lem:transformsolution})
  proves that for all 
    $x\in\R^d$,
    $t\in[0,T]$,
    $\omega\in\Omega$
  it holds that
  \begin{equation}
    \label{eq:lem2.Z}
    Z^x(t,\omega)
    =
    x+\int_0^t\mu(Z^x(s,\omega))\,\diff s+\sigma W(t,\omega).
  \end{equation}
  Furthermore, note that
    the chain rule
  implies that for all
    $x\in\R^d$
  it holds that
  \begin{equation}
    \label{eq:lem2.dmu}
    \mu'(x)
    =
    B\nu'(B^{-1}(x-v))B^{-1}
    .
  \end{equation}
  In addition, observe that
    the assumption that $A$ is an orthogonal matrix
  ensures that for all
    $x\in\R^d$
  it holds that
  \begin{equation}
    \label{eq:lem2.B}
    \norm{Bx}=\norm{\norm\delta Ax}=\norm\delta\norm{Ax}=\norm\delta\norm x.
  \end{equation}
  Combining
    this
  with 
    \eqref{eq:lem2.dmu}
    and \cref{it:lem2.lem1.1}
  proves that for all 
    $x,h\in\R^d$ 
  it holds that
  \begin{equation}
    \label{eq:lem2.mu}
    \begin{split}
    \norm{\mu'(x)h}
    &=
    \norm{B\nu'(B^{-1}(x-v))B^{-1}h}
    =
    \norm\delta \norm{\nu'(B^{-1}(x-v))B^{-1}h}
    \\&\leq
    \norm\delta \varkappa \bp{1+\norm{B^{-1}(x-v)}^\varkappa}\norm{B^{-1}h}
    \\&=
    \norm\delta \varkappa \bp{1+\norm\delta^{-1} \norm{x-v}^\varkappa}\norm\delta^{-1}\norm{h}
    \\&=
    \varkappa \bp{1+\norm\delta^{-1} \norm{x-v}^\varkappa}\norm{h}
    \\&\leq
    \varkappa \bp{1+\norm\delta^{-1} 2^{\varkappa}(\norm x^\varkappa+\norm v^\varkappa)}\norm{h}
    \\&=
    \varkappa \bp{1+\norm\delta^{-1} 2^{\varkappa}\norm v^{\varkappa}+\norm\delta^{-1}2^\varkappa\norm x^\varkappa}\norm{h}
    \\&\leq
    \varkappa \bp{1+\norm\delta^{-1} 2^{\varkappa}\max\{1,\norm v^\varkappa\}}\bp{1+\norm x^\varkappa}\norm{h}
    .
    \end{split}
  \end{equation}
  In the next step let
    $\kappa\in[1,\infty)$
  satisfy
  \begin{equation}
    \kappa=2\varkappa \bp{1+\norm\delta^{-1} 2^{\varkappa}\max\{1,\norm v^\varkappa\}}
    .
  \end{equation}
  Note that
    \eqref{eq:lem2.mu}
  demonstrates that for all 
    $x,h\in\R^d$ 
  it holds that
  \begin{equation}
    \label{eq:lem2.mu2}
    \begin{split}
    \norm{\mu'(x)h}
    &\leq
    \varkappa \bp{1+\norm\delta^{-1} 2^{\varkappa}\max\{1,\norm v^\varkappa\}}\bp{2+\norm x^\kappa}\norm{h}
    \\&\leq
    2\varkappa \bp{1+\norm\delta^{-1} 2^{\varkappa}\max\{1,\norm v^\varkappa\}}\bp{1+\norm x^\kappa}\norm{h}
    \\&=
    \kappa\bp{1+\norm x^\kappa}\norm{h}
    .
    \end{split}
  \end{equation}
  Next let $V\colon \R^d\to[0,\infty)$ 
    satisfy for all
      $x\in\R^d$
    that
    \beq
    \label{eq:lem2.defV}
      V(x)=\norm\delta U(B^{-1}(x-v))+\norm{v}.
    \eeq
  Observe that 
    the chain rule
  implies that for all
    $x\in\R^d$
  it holds that
  \begin{equation}
    V'(x)
    =
    \norm\delta U'(B^{-1}(x-v))B^{-1}
    .
  \end{equation}
  Hence, we obtain that for all
    $x\in\R^d$,
    $z\in\R$
  it holds that
  \ba
    \label{eq:lem2.V}
    V'(x)\mu(x+\sigma z)
    &=
    \norm\delta U'(B^{-1}(x-v))B^{-1}\mu(x+\sigma z)
    \\&=
    \norm\delta U'(B^{-1}(x-v))B^{-1}\bbr{B\nu(B^{-1}(x+\sigma z-v))}
    \\&=
    \norm\delta U'(B^{-1}(x-v))\nu(B^{-1}(x-v)+\rho z))
    \\&\leq
    \norm\delta\varkappa(1+\abs z)U(B^{-1}(x-v))
    \\&=
    \varkappa(1+\abs z)(V(x)-\norm v)
    \\&\leq
    \varkappa(1+\abs z)V(x)
    \leq
    \kappa(1+\abs z)V(x)
    .
  \ea
  Furthermore, note that
    \eqref{eq:lem2.defV}
    and \cref{it:lem2.lem1.1}
  ensure that for all
    $x\in\R^d$
  it holds that
  \begin{equation}
  \begin{split}
    V(x)
    &\geq
    \norm\delta\norm{B^{-1}(x-v)}+\norm v
    =
    \norm\delta\norm\delta^{-1}\norm{x-v}+\norm v
    \\&=
    \norm{x-v}+\norm v
    \geq
    \norm x
    .
  \end{split}
  \end{equation}
  Combining
    this,
    \eqref{eq:lem2.mu2},
    and \eqref{eq:lem2.V}
  establishes \cref{it:lem2.1}.
    This
    and \cite[Lemma~5.4]{JentzenKuckuckMuellerGronbachYaroslavtseva2019}
  shows that there exist unique stochastic processes 
    $X^x\colon [0,T]\times\Omega\to\R^d$, $x\in\R^d$, with continuous sample paths
  which satisfy for all 
    $x\in\R^d$,
    $t\in[0,T]$,
    $\omega\in\Omega$
  that
  \beq
    \label{eq:lem2.X}
    X^x(t,\omega)=x+\int_0^t\mu(X^x(s,\omega))\,\diff s+\sigma W(t,\omega).
  \eeq
  Therefore, we obtain \cref{it:lem2.2}.
  In addition, observe that
    \eqref{eq:lem2.X}
    and \eqref{eq:lem2.Z}
  imply that for all
    $x\in\R^d$,
    $t\in[0,T]$,
    $\omega\in\Omega$
  it holds that
  \begin{equation}
    X^x(t,\omega)=Z^x(t,\omega).
  \end{equation}
    This,
    \eqref{eq:lem2.defA},
    \eqref{eq:lem2.defB},
    and \eqref{eq:lem2.defX} 
  prove that for all
    $r\in\R$,
    $t\in(\tau,T)$,
    $\omega\in\Omega$
  it holds that
  \begin{equation}
    \begin{split}
    X^v(t,\omega)-X^{v+r\delta}(t,\omega)
    &=
    \bbr{BY^{B^{-1}(v-v)}(t,\omega)+v}
    -
    \bbr{BY^{B^{-1}(v+r\delta-v)}(t,\omega)+v}
    \\&=
    B\bp{Y^{0}(t,\omega)-Y^{r\norm{\delta}^{-1}A^{-1}\delta}(t,\omega)}
    \\&=
    B\bp{Y^{0}(t,\omega)-Y^{ru}(t,\omega)}
    .
    \end{split}
  \end{equation}
  Combining
    this
  with
    \cref{it:lem2.lem1.3}
  shows that for all
    $t\in(\tau,T)$ 
  there exists
    $c\in(0,\infty)$ 
  such that
    for all
      $r\in(0,\nicefrac1e]$
    it holds that
    \ba
      \bEE{ \norm{X^v(t)-X^{v+r\delta}(t)} }
      &=
      \bEE{ \norm{B\bp{Y^{0}(t)-Y^{ru}(t)}} }
      \\&=
      \norm\delta\bEE{ \norm{Y^{0}(t)-Y^{ru}(t)} }
      \\&\geq
      \norm\delta\exp\bp{ -c\,\abs{\ln(r)}^{2/n} }
      .
    \ea
  This establishes \cref{it:lem2.3}.
  The proof of \cref{lem:lem2} is thus completed.
\end{proof}

\subsection[On SDEs with non-locally H\"older continuous dependence on the initial values]{On solutions to SDEs with non-locally H\"older continuous dependence on the initial values}
\label{subsec:final}

\begin{lemma}
\label{lem:hoeldercomp}
  Let $c,R,\alpha\in(0,\infty)$, 
  $K\in[0,\infty)$, 
  $\beta\in(0,1)$
  satisfy
  \beq
    \label{eq:hoeldercomp.Kdef}
    K=\inf\bbbbr{ 
      \{1\} \cup \biggl\{
        \frac{\exp(-c\,\abs{\ln(r)}^\beta)}{r^\alpha}
        \colon 
        r\in\bigl[\exp\bp{-\br{\alpha^{-1}c}^{\frac1{1-\beta}}},\infty\bigr)\cap(0,R]
      \biggr\} 
    }
    .
  \eeq
  Then
  \begin{enumerate}[label=(\roman{*})]
    \item \label{it:hoeldercomp.1}
      it holds for all 
        $r\in \bigl(0,\exp\bp{-[\alpha^{-1}c]^{\frac1{1-\beta}}}\bigr]$
      that
        $\exp(-c\,\abs{\ln(r)}^\beta)\geq r^\alpha$
      and
    \item \label{it:hoeldercomp.2}
      it holds for all $r\in(0,R]$ that 
        $K>0$ 
        and $\exp(-c\,\abs{\ln(r)}^\beta)\geq K r^\alpha$.
  \end{enumerate}
\end{lemma}
\begin{proof}[Proof of \cref{lem:hoeldercomp}]
  Throughout this proof let $C\in(0,\infty)$
  \intrtype{be the real number which satisfies }%
  \intrtypen{satisfy }%
  \beq
    C=\exp\bp{-\br{\alpha^{-1} c}^{\frac{1}{1-\beta}}}.
  \eeq
  Note that
    the fact that $C<1$
  shows that for all
    $r\in(0,C]$ 
  it holds that
  \ba
    \ln(r)
    &=
    -\abs{\ln(r)}
    =
    -\abs{\ln(r)}^{1-\beta}\abs{\ln(r)}^\beta
    \leq
    -\abs{\ln(C)}^{1-\beta}\abs{\ln(r)}^{\beta}
    \\&=
    -\bigl\lvert-[\alpha^{-1}c]^{\frac{1}{1-\beta}}\bigr\rvert^{1-\beta}\abs{\ln(r)}^{\beta}
    =-\alpha^{-1}c\,\abs{\ln(r)}^\beta
    .
  \ea
  Hence, we obtain that for all
    $r\in(0,C]$
  it holds that
  \beq
    \label{eq:hoeldercomp.1}
    r^\alpha=\exp(\alpha\ln(r))\leq \exp(-c\,\abs{\ln(r)}^\beta)
    .
  \eeq
  Next observe that
    \eqref{eq:hoeldercomp.Kdef}
  implies that for all
    $r\in[C,\infty)\cap(0,R]$
  it holds that
  \beq
    \label{eq:hoeldercomp.2}
    Kr^\alpha\leq \exp(-c\,\abs{\ln(r)}^\beta).
  \eeq
  Furthermore, note that
    the fact that $(0,\infty)\ni r\mapsto \exp(-c\,\abs{\ln(r)}^\beta)r^{-\alpha}\in(0,\infty)$ is a continuous function
    and the fact that $[C,\infty)\cap(0,R]$ is a compact set
  ensure that $K>0$.
  Combining 
    this 
  with
    \eqref{eq:hoeldercomp.1},
    \eqref{eq:hoeldercomp.2},
    and the fact that $K\leq 1$
  establishes
    \cref{it:hoeldercomp.1,it:hoeldercomp.2}.
  The proof of \cref{lem:hoeldercomp} is thus completed.
\end{proof}

\begin{theorem}
\label{thm:final}
  Let $m\in\N$,
  $d\in\{5,6,\dots\}$,
  $T\in(0,\infty)$,
  $\tau\in(0,T)$,
  $v\in\R^d$,
  $\delta\in\R^d\setminus\{0\}$
  let $\norm{\cdot}\colon\R^d\to[0,\infty)$ be the standard norm on $\R^d$,
  let $\nnorm{\cdot}\colon\R^m\to[0,\infty)$ be a norm,
  let $(\Omega,\mc F,\PP)$ be a probability space,
  and let $W\colon[0,T]\times\Omega\to\R^m$ be a standard Brownian motion
    with continuous sample paths.
  Then there exist
    $\mu\in C^\infty(\R^d,\R^d)$,
    $\sigma\in\R^{d\times m}$,
    $V\in C^\infty(\R^d,[0,\infty))$,
    $\kappa\in(0,\infty)$
  such that
  \begin{enumerate}[label=(\roman{enumi}),ref=(\roman{enumi})]
    \item \label{it:thmfinal:1} 
    it holds for all
      $x,h\in\R^d$, 
      $z\in\R^m$ 
    that
      $\norm{\mu'(x)h}\leq \kappa \bp{1+\norm x^\kappa} \norm h$,
      $V'(x)\mu(x+\sigma z)\leq \kappa(1+\nnorm{z})V(x)$,
      and $\norm x\leq V(x)$,
    \item \label{it:thmfinal:2}
      there exist unique stochastic processes
      $X^x\colon [0,T]\times\Omega\to\R^d$, $x\in\R^d$, with continuous sample paths
      such that for all 
        $x\in\R^d$,
        $t\in[0,T]$,
        $\omega\in\Omega$
      it holds that
      \beq
        X^x(t,\omega)=x+\int_0^t\mu(X^x(s,\omega))\,\diff s+\sigma W(t,\omega)
        ,
      \eeq
    \item\label{it:thmfinal:3}
      it holds for all 
        $\omega\in\Omega$ 
      that 
        $\bp{[0,T]\times\R^d\ni(t,x)\mapsto X^x(t,\omega)\in\R^d}\in C^{0,1}([0,T]\times\R^d,\R^d)$,
    \item\label{it:thmfinal:4}
      it holds for all
        $x,h\in\R^d$,
        $t\in[0,T]$,
        $\omega\in\Omega$
      that
      \beq
        \bp{{\textstyle\frac\partial{\partial x}} X^x(t,\omega)}(h)
        =
        h+\int_0^t \mu'(X^x(s,\omega))\bp{\bp{\tfrac\partial{\partial x} X^x(s,\omega)}(h)} \,\diff s
        ,
      \eeq
    \item\label{it:thmfinal:50}
    it holds for all 
      $R,r\in(0,\infty)$
    that
      $\Omega\ni\omega\mapsto \sup_{x\in[-R,R]^d}\,\sup_{t\in[0,T]} \bp{\norm{X^x(t,\omega)}^{r}}\in[0,\infty]$
      is an $\mc F/\mc B([0,\infty])$-measurable function,
    \item\label{it:thmfinal:5}
      it holds for all 
        $R,r\in(0,\infty)$
      that
      \beq
        \bbbEE{ \sup_{x\in[-R,R]^d}\,\sup_{t\in[0,T]} \bp{\norm{X^x(t)}^{r}} }<\infty
        ,
      \eeq
    \item\label{it:thmfinal:6}
      it holds for all 
        $R,q\in(0,\infty)$
      that there exists 
        $c\in(0,\infty)$ 
      such that for all 
        $x,y\in[-R,R]^d$ with $0<\norm{x-y}\neq 1$ 
      it holds that
      \beq
        \sup_{t\in[0,T]}\bEE{\norm{X^x(t)-X^y(t)}}\leq c\,\abs{\ln(\norm{x-y})}^{-q}
        ,
      \eeq
      \item\label{it:thmfinal:71}
      there exists
        $K\in(0,\infty)$
      such that for all 
        $t\in(\tau,T)$ 
      there exists
        $c\in(0,\infty)$ 
      such that for all
        $w\in\{v+r\delta\colon r\in(0,1]\}$
      it holds that
      \beq
        K \exp\bp{ -c\,\abs{\ln(\norm{v-w})}^{2/n} }
        \leq 
        \bEE{ \norm{X^v(t)-X^w(t)} } 
        ,
      \eeq
      and
    \item\label{it:thmfinal:8}
      it holds for all 
        $t\in(\tau,T)$,
        $\alpha\in(0,\infty)$
      that there exists 
        $c\in(0,\infty)$ 
      such that for all
        $w\in\{v+r\delta\colon r\in[0,1]\}$
      it holds that
      \beq
      \label{eq:finalnonhoelder}
        c\,\norm{v-w}^\alpha
        \leq
        \bEE{ \norm{X^v(t)-X^w(t)} } 
        .
      \eeq
  \end{enumerate}
\end{theorem}
\begin{proof}[Proof of \cref{thm:final}]
  \newcommand{\vmu}{\mu}
  \newcommand{\vrho}{\rho}
  \newcommand{\vokappa}{\varkappa}
  \newcommand{\vV}{V}
  Throughout this proof
    let $K\in(0,\infty)$ satisfy
    \beq
      K=\max\Biggl\{1,\,\sup\nolimits_{z=(z_1,z_2,\dots,z_m)\in\R^m\setminus\{0\}} \frac{\sum_{i=1}^m\abs{z_i}}{\nnorm{z}}\Biggr\}
      .
    \eeq
    and let $W_i\colon[0,T]\times\Omega\to\R$, $i\in\{1,2,\dots,m\}$,
      satisfy for all
        $t\in[0,T]$,
        $\omega\in\Omega$
      that
      \beq
        W(t,\omega)
        =
        \bp{W_1(t,\omega),W_2(t,\omega),\dots,W_m(t,\omega)}
        .
      \eeq
  Note that \cref{lem:lem2}
    (with
      $d\is d$,
      $T\is T$,
      $\tau\is\tau$,
      $v\is v$,
      $\delta\is e\delta$,
      $\norm{\cdot}\is\norm{\cdot}$,
      $(\Omega,\mc F,\PP)\is(\Omega,\mc F,\PP)$,
      $W\is W_1$
    in the notation of \cref{lem:lem2})
  establishes that there exist
    $\vmu\in C^\infty(\R^d,\R^d)$,
    $\vrho=(\vrho_1,\vrho_2,\dots,\vrho_d)\in\R^d$,
    $\vV\in C^\infty(\R^d,[0,\infty))$,
    $\vokappa\in(0,\infty)$
  which satisfy that
  \begin{enumerate}[label=(\Alph{enumi}),ref=(\Alph{enumi})]
    \item \label{it:muVcond} 
      it holds for all
        $x,h\in\R^d$, 
        $z\in\R$ 
      that
        $\norm{\vmu'(x)h}\leq \vokappa \bp{1+\norm x^\vokappa} \norm h$,
        $\vV'(x)\vmu(x+\vrho z)\leq \vokappa(1+\abs{z})\vV(x)$,
        and $\norm x\leq \vV(x)$,
    \item
      there exist unique stochastic processes
        $X^x\colon [0,T]\times\Omega\to\R^d$, $x\in\R^d$, with continuous sample paths
      such that for all 
        $x\in\R^d$,
        $t\in[0,T]$,
        $\omega\in\Omega$
      it holds that
      \beq
        \label{eq:Xdef}
        X^x(t,\omega)=x+\int_0^t\vmu(X^x(s,\omega))\,\diff s+\vrho W_1(t,\omega)
        ,
      \eeq
      and
    \item \label{it:strongnonhoelder} 
      it holds for all 
        $t\in(\tau,T)$ 
      that 
        there exists
          $c\in(0,\infty)$ 
        such that
          for all
            $w\in\{v+r\delta\colon r\in(0,1]\}$
          it holds that
          \beq
            e\norm\delta\exp\bp{ -c\,\abs{\ln(\norm{v-w})}^{2/n} }
            \leq 
            \bEE{ \norm{X^v(t)-X^w(t)} } 
            .
          \eeq
  \end{enumerate}
  In the next step let 
    $\kappa\in(0,\infty)$
    and
    $\sigma=(\sigma_{i,j})_{i\in\{1,2,\dots,d\},\,j\in\{1,2,\dots,m\}}\in\R^{d\times m}$ 
    satisfy for all
      $i\in\{1,2,\dots,d\}$,
      $j\in\{1,2,\dots,m\}$
    that
    \beq
    \label{eq:sigmadef}
    \kappa=2K\vokappa
    \qqandqq
      \sigma_{i,j}
      =
      \begin{cases}
        \vrho_i & \colon j=1 \\
        0 & \colon j>1.
      \end{cases}
    \eeq
  Observe that 
    \cref{it:muVcond}
  ensures that for all 
    $x\in\R^d$,
    $z=(z_1,z_2,\dots,z_m)\in\R^m$
  it holds that
  \ba
    \vV'(x)\vmu(x+\sigma z)
    &=
    \vV'(x)\vmu(x+\vrho z_1)
    \leq
    \vokappa(1+\abs{z_1})V(x)
    \\&\leq
    \vokappa(1+K\nnorm{z})V(x)
    \leq
    \kappa (1+\nnorm{z})V(x)
    .
  \ea
  Furthermore, note that
    \cref{it:muVcond}
  shows that for all
    $x,h\in\R^d$,
    $z\in\R$
  it holds that
  \begin{equation}
    \norm{\mu'(x)h}
    \leq
    \vokappa(2+\norm x^\kappa)\,\norm h
    \leq
    2\vokappa(1+\norm x^\kappa)\,\norm h
    \leq
    \kappa(1+\norm x^\kappa)\,\norm h
    .
  \end{equation}
    This
    and \cref{it:muVcond}
  prove \cref{it:thmfinal:1}.
  In the next step we observe that
    \eqref{eq:sigmadef}
  implies that for all
    $t\in[0,T]$,
    $\omega\in\Omega$
  it holds that
  \begin{equation}
    \sigma W(t,\omega)
    =
    \rho W_1(t,\omega)
    .
  \end{equation}
    This 
    and \eqref{eq:Xdef}
  establish \cref{it:thmfinal:2}.
  Next note that 
    \cite[Lemma~5.4]{JentzenKuckuckMuellerGronbachYaroslavtseva2019} 
  proves \cref{it:thmfinal:3,it:thmfinal:4}.
  In addition, observe that
    \cite[Lemma~6.6]{JentzenKuckuckMuellerGronbachYaroslavtseva2019}
  demonstrates \cref{it:thmfinal:50,it:thmfinal:5}.
  Moreover, note that
    \cite[Lemma~8.4]{JentzenKuckuckMuellerGronbachYaroslavtseva2019}
  establishes \cref{it:thmfinal:6}.
  Next observe that
    \cref{it:strongnonhoelder}
  implies
    \cref{it:thmfinal:71}.
  Combining
    \cref{it:thmfinal:71}
    with \cref{lem:hoeldercomp}
  proves \cref{it:thmfinal:8}.
  The proof of \cref{thm:final} is thus completed.
\end{proof}


\bibliographystyle{acm}
\bibliography{mgy}

\begin{thebibliography}{10}

\bibitem{CL14}
{\sc Chen, X., and Li, X.-M.}
\newblock Strong completeness for a class of stochastic differential equations
  with irregular coefficients.
\newblock {\em Electron. J. Probab. 19}, 91 (2014), 1--34.

\bibitem{CoxHutzenthalerJentzen}
{\sc Cox, S., Hutzenthaler, M., and Jentzen, A.}
\newblock Local {L}ipschitz continuity in the initial value and strong
  completeness for nonlinear stochastic differential equations.
\newblock {\em arXiv:1309.5595\/} (2013), 84 pages.
\newblock Revision requested from Mem.\ Amer.\ Math.\ Soc.

\bibitem{FIZ07}
{\sc Fang, S., Imkeller, P., and Zhang, T.}
\newblock Global flows for stochastic differential equations without global
  {L}ipschitz conditions.
\newblock {\em Ann. Probab. 35}, 1 (2007), 180--205.

\bibitem{HHJ15}
{\sc Hairer, M., Hutzenthaler, M., and Jentzen, A.}
\newblock Loss of regularity for {K}olmogorov equations.
\newblock {\em Ann. Probab. 43}, 2 (2015), 468--527.

\bibitem{HairerMattingly}
{\sc {Hairer}, M., and {Mattingly}, J.~C.}
\newblock {Ergodicity of the 2D Navier-Stokes equations with degenerate
  stochastic forcing.}
\newblock {\em {Ann. Math. (2)} 164}, 3 (2006), 993--1032.

\bibitem{HuddeHutzenthalerJentzenMazzonetto}
{\sc Hudde, A., Hutzenthaler, M., Jentzen, A., and Mazzonetto, S.}
\newblock On the {I}t{\^o}-{A}lekseev-{G}r{\"o}bner formula for stochastic
  differential equations.
\newblock {\em arXiv:1812.09857\/} (2018), 29 pages.

\bibitem{HHM19}
{\sc {Hudde}, A., {Hutzenthaler}, M., and {Mazzonetto}, S.}
\newblock {A stochastic Gronwall inequality and applications to moments, strong
  completeness, strong local Lipschitz continuity, perturbations}.
\newblock {\em arXiv:1903.08727\/} (2019), 21 pages.

\bibitem{HutzenthalerJentzen}
{\sc Hutzenthaler, M., and Jentzen, A.}
\newblock On a perturbation theory and on strong convergence rates for
  stochastic ordinary and partial differential equations with non-globally
  monotone coefficients.
\newblock {\em arXiv:1401.0295\/} (2014), 41 pages.
\newblock To appear in Ann.\ Probab.

\bibitem{hutzenthaler2019strong}
{\sc Hutzenthaler, M., Jentzen, A., Lindner, F., and Pu{\v z}nik, P.}
\newblock Strong convergence rates on the whole probability space for
  space-time discrete numerical approximation schemes for stochastic {B}urgers
  equations.
\newblock {\em arXiv:1911.01870\/} (2019), 60 pages.

\bibitem{JentzenKuckuckMuellerGronbachYaroslavtseva2019}
{\sc {Jentzen}, A., {Kuckuck}, B., {M{\"u}ller-Gronbach}, T., and
  {Yaroslavtseva}, L.}
\newblock {On the strong regularity of degenerate additive noise driven
  stochastic differential equations with respect to their initial values}.
\newblock {\em arXiv:1904.05963\/} (2019), 59 pages.

\bibitem{JMGY16}
{\sc Jentzen, A., M\"{u}ller-Gronbach, T., and Yaroslavtseva, L.}
\newblock On stochastic differential equations with arbitrary slow convergence
  rates for strong approximation.
\newblock {\em Commun. Math. Sci. 14}, 6 (2016), 1477--1500.

\bibitem{Kry99}
{\sc Krylov, N.~V.}
\newblock On {K}olmogorov's equations for finite-dimensional diffusions.
\newblock In {\em Stochastic {PDE}'s and {K}olmogorov equations in infinite
  dimensions ({C}etraro, 1998)}, vol.~1715 of {\em Lecture Notes in Math.}
  Springer, Berlin, 1999, pp.~1--63.

\bibitem{Li94}
{\sc Li, X.-M.}
\newblock Strong {$p$}-completeness of stochastic differential equations and
  the existence of smooth flows on noncompact manifolds.
\newblock {\em Probab. Theory Related Fields 100}, 4 (1994), 485--511.

\bibitem{LS11}
{\sc Li, X.-M., and Scheutzow, M.}
\newblock Lack of strong completeness for stochastic flows.
\newblock {\em Ann. Probab. 39}, 4 (2011), 1407--1421.

\bibitem{LiuRoeckner}
{\sc Liu, W., and R\"{o}ckner, M.}
\newblock {\em Stochastic partial differential equations: an introduction}.
\newblock Universitext. Springer, Cham, 2015.

\bibitem{SS17}
{\sc Scheutzow, M., and Schulze, S.}
\newblock Strong completeness and semi-flows for stochastic differential
  equations with monotone drift.
\newblock {\em J. Math. Anal. Appl. 446}, 2 (2017), 1555--1570.

\bibitem{Zhang10}
{\sc Zhang, X.}
\newblock Stochastic flows and {B}ismut formulas for stochastic {H}amiltonian
  systems.
\newblock {\em Stochastic Process. Appl. 120}, 10 (2010), 1929--1949.

\end{thebibliography}

\end{document}